\documentclass[a4paper, 12pt]{article}
\usepackage[a4paper, left=2.5cm, right=2.5cm, top=2.5cm, bottom=2cm]{geometry}
\usepackage[utf8]{inputenc}
\usepackage{lmodern} 
\usepackage[english]{babel}
\usepackage{ngerman}
\usepackage{amsmath}
\numberwithin{equation}{section}
\usepackage{amsfonts}
\usepackage{amsthm}
\usepackage{amssymb}
\usepackage[ruled,vlined]{algorithm2e}
\usepackage{graphicx}
\usepackage{hyperref}
\usepackage{algorithmic}
\usepackage{url}

\newtheorem{theorem}{Theorem}[section] 

\newtheorem{remark}[theorem]{Remark}
\newtheorem{assumption}[theorem]{Assumption}
\newtheorem{lemma}[theorem]{Lemma}
\newtheorem{corollary}[theorem]{Corollary}

\usepackage{accents}

\title{The ensemble Kalman filter for rare event estimation}
\author{F. Wagner, I. Papaioannou, E. Ullmann}
\date{\today}

\begin{document}

\maketitle

\begin{abstract}
We present a novel sampling-based method for estimating probabilities of rare or failure events. Our approach is founded on the Ensemble Kalman filter (EnKF) for inverse problems. Therefore, we reformulate the rare event problem as an inverse problem and apply the EnKF to generate failure samples. To estimate the probability of failure, we use the final EnKF samples to fit a distribution model and apply Importance Sampling with respect to the fitted distribution. This leads to an unbiased estimator if the density of the fitted distribution admits positive values within the whole failure domain. To handle multi-modal failure domains, we localise the covariance matrices in the EnKF update step around each particle and fit a mixture distribution model in the Importance Sampling step. For affine linear limit-state functions, we investigate the continuous-time limit and large time properties of the EnKF update. We prove that the mean of the particles converges to a convex combination of the most likely failure point and the mean of the optimal Importance Sampling density if the EnKF is applied without noise. We provide numerical experiments to compare the performance of the EnKF with Sequential Importance Sampling.
\vspace{0.5cm}
\end{abstract}

\section{Introduction}\label{Sec: Introduction}
Estimating the probability of failure is a frequent and crucial task in reliability analysis and risk management \cite{Agarwal17, Morio15}. Failure of a system is determined by the outcome of a \emph{limit-state function} (LSF). By convention, if the outcome is larger than zero, the given state keeps the system in a safe mode. For a negative outcome, the state leads to failure. The \emph{probability of failure} is defined as the probability mass of all failure states. Since failure probabilities are likely to be small, estimation of the failure probability requires the simulation of \emph{rare events}. 
\\Often, the evaluation of the LSF requires the evaluation of a computational expensive model, a partial differential equation, which makes crude Monte Carlo sampling \cite{Fishman96,Rubinstein16} prohibitive. Variance reduction techniques like \emph{Subset Simulation} (SuS) \cite{Au01, Au14}, \emph{Sequential Importance Sampling} (SIS) \cite{Papaioannou16, Wagner20} or the \emph{cross-entropy} based Importance Sampling (IS) method \cite{Kroese13, Papaioannou19, Wang16} have been developed to reduce computational costs while preserving an accurate estimate. In \emph{line sampling} \cite{Angelis15,Koutsourelakis04,Rackwitz01}, sampling is performed on a hyperplane perpendicular to an important direction. Alternative to sampling methods, approximation methods, like the \emph{first} and \emph{second order reliability method} \cite{Melchers18} (FORM/SORM), determine the \emph{most likely failure point} (MLFP) and approximate the surface of the failure domain. In our novel approach, we investigate the \emph{Ensemble Kalman filter} (EnKF) for inverse problems proposed by \cite{Iglesias2013, Schillings16}, which is also known as \emph{ensemble Kalman inversion} (EKI), and apply it to rare event estimation. The EnKF is a sampling-based method. 
\\The Kalman filter \cite{Kalman60} has been originally proposed for data assimilation problems. If the observation operator and dynamic are linear, and the initial and noise distribution are Gaussian, the Kalman filter is exact. The Extended Kalman filter \cite[Section 4.2.2]{Law15} is applied for nonlinear dynamics. In this case, the dynamic is approximated via its linearization and the Kalman update is applied. However, this requires derivative information, which might be costly to obtain. The EnKF \cite{Evensen06} approximates the derivative via an ensemble of particles. Thus, the EnKF does not require derivative information. However, the EnKF is only asymptotically exact if the dynamic is linear. To apply the EnKF to systems with multimodal distributions, the authors of \cite{Dovera11, Li16, Smith07} propose to fit a GM distribution in each EnKF update step and to update the particles belonging to each mixture term separately. 
\\Compared to SIS and SuS, the EnKF has several advantages which are our motivation to implement the EnKF for rare event estimation. The EnKF is easier to implement since a \emph{Markov chain Monte Carlo} (MCMC) \cite{Hastings70} algorithm is not required. Thus, the EnKF contains fewer hyperparameters which have to be tuned. Moreover, it requires no burn-in and no tuning for the optimal acceptance rate has to be performed. Since the EnKF particles are equally weighted, no computational costs are wasted for degenerated samples with very small weights or for rejected samples. Moreover, the EnKF is applicable for high-dimensional problems.
\\In recent years, several theoretical properties of the EnKF have been studied. In \cite{Iglesias2013}, it is shown that the EnKF satisfies the \emph{subspace property}, i.e., the particles stay for all iterations within the subspace spanned by the initial ensemble. The authors of \cite{Bloemker19, Schillings16, Schillings18} investigate the continuous time limit of the EnKF update step, which results in a coupled system of \emph{stochastic differential equations} (SDEs). For a fixed ensemble size, linear system, and considering the limit $t\rightarrow\infty$, the ensemble collapses to its mean value. The authors of \cite{Garbuno20, Herty19} study the mean and covariance of the EnKF particles for an infinite ensemble in the linear Gaussian setting. The distribution of the EnKF particles is equal to the posterior distribution for $t=1$ in the continuous time limit \cite{Garbuno20}. The limit $t\rightarrow\infty$ yields again ensemble collapse. In the nonlinear case, the work of \cite{Ernst15} shows that the distribution of the ensemble particles does not converge to the posterior distribution even for an infinite ensemble. Instead, the particles approximate the distribution of a so-called \emph{analysis variable}.
\\To apply the EnKF algorithm for rare event simulation, we formulate the rare event problem as an inverse problem via an auxiliary LSF. The auxiliary LSF is the concatenation of the rectified linear unit (ReLU) and the original LSF. We apply the EnKF for inverse problems to this reformulation. Since the distribution of the analysis variable differs from the posterior distribution in general, we fit a distribution model with the final EnKF particles and apply IS with respect to the fitted distribution to estimate the probability of failure. This procedure is similar to \cite[Algorithm 3.1]{Kroese13}. In particular, we apply the \emph{Gaussian mixture} (GM) and the \emph{von Mises--Fisher--Nakagami mixture} (vMFNM) distribution model, which have also been used in the context of the cross-entropy method \cite{Papaioannou19}. Due to the properties of IS, we achieve an unbiased estimator for the probability of failure if the density of the fitted distribution admits positive values within the whole failure domain. Additionally, we apply an adaptive approach to determine the sequence of discretization steps. This approach is similar to adaptive \emph{Sequential Monte Carlo} \cite{Beskos13} and is already applied to the EnKF in \cite{Iglesias18}. We demonstrate that the sequence of the resulting EnKF densities is a piecewise smooth approximation of the discontinuous optimal IS density. 
\\Since the EnKF particles always approximate the distribution of one single mode, we apply a multi-modal strategy to handle multi-modal failure domains. In our work, we consider the approach of \cite{Reich19}, where the means and covariance matrices of the EnKF update are localised around each particle. Therefore, the particles evolve individually (within their nearest neighbourhood) to distinct failure modes. The approach of \cite{Reich19} requires the choice of a localization parameter, which we avoid through employing a clustering algorithm.
\\Moreover, we translate some of the already observed theoretical properties for the EnKF to the rare event setting. We derive the continuous time limit of the particle dynamic for affine linear LSFs. If the EnKF is applied without noise in the data space, we prove that the ensemble mean converges to a convex combination of the MLFP and of the mean of the optimal IS density.
\\The manuscript is structured as follows. Section~\ref{Section Settings} reviews the general setting of rare event estimation and introduces an equivalent formulation as an inverse problem. Thereafter, \emph{Bayesian inverse problems} (BIPs) and the EnKF for inverse problems are discussed. Section~\ref{Section EnKF} contains the formulation of the EnKF for estimating the probability of failure. Section~\ref{Sec EnKF for affine linear LSFs} shows theoretical properties of the EnKF for affine linear LSFs. The proofs are given in the Appendix~\ref{Appendix}. In Section~\ref{chapter numerical experiments}, the EnKF is applied to numerical experiments and its performance is compared with SIS. We end this manuscript with a conclusion in Section~\ref{Section Conclusion}.

\section{Problem Setting}\label{Section Settings}
We start by defining failure events and the probability of failure. We discuss the standard formulation and introduce an alternative formulation that draws an analogy to BIPs. We introduce BIPs and the well-known Bayes' theorem. Moreover, we discuss the EnKF for inverse problems and its theoretical properties.

\subsection{Rare event estimation}\label{Sec rare events}
The following notation is based on \cite{Papaioannou15, Papaioannou16}. It is common to define failure events via an LSF that distinguishes safe and failure states. Let $\left(\Omega, \mathcal{A}, \mathbb{P}\right)$ be a probability space. Given is an LSF $G: \mathbb{R}^d\rightarrow \mathbb{R}$, which models the performance of a system. Importantly, the LSF often depends on a computationally intensive numerical model of the system. The state $u\in\mathbb{R}^d$ leads to failure if $G(u)\le 0$; otherwise, $u$ is a safe state.
\\The input state $u\in\mathbb{R}^d$ is a realization of an $\mathbb{R}^d$-valued random variable $U:\Omega\rightarrow\mathbb{R}^d$. We will sometimes use the equivalent definition of a failure event by $\{\omega\in\Omega: G(U(\omega))\le 0\}$ to emphasize that the outcome of the LSF depends on an event $\omega\in\Omega$. The random variable $U$ is distributed according to the \emph{probability density function} (pdf) $\mu_0: \mathbb{R}^d\rightarrow [0,\infty[$. 
\\The goal is to estimate the probability mass of the states $u\in\mathbb{R}^d$, or equivalently, of events $\omega\in\Omega$, which lead to failure. This probability is called \emph{probability of failure} and is denoted by $P_f$. Indeed, the probability of failure is given by
\begin{align}
P_f := \mathbb{P}\left(\{\omega\in\Omega : G(U(\omega)) \le 0\}\right) = \int_{u\in\mathbb{R}^d} I(G(u)\le 0) \mu_0(u)\mathrm{d}u,\label{Probability of failure}
\end{align}
where $I$ denotes the indicator function. Since the \emph{failure domain} $\{u\in\mathbb{R}^d : G(u)\le 0\}$ is a priori unknown, estimating $P_f$ accurately is a nontrivial task. 
\\In IS \cite{Agapiou17,Rubinstein16}, the integral in~\eqref{Probability of failure} is expressed as an expectation with respect to an IS density $p:\mathbb{R}^d\rightarrow\mathbb{R}$
\begin{align*}
P_f = \int_{u\in\mathbb{R}^d} I(G(u)\le0)w(u)p(u)\mathrm{d}u = \mathbb{E}_p[I(G(u)\le 0)w(u)],
\end{align*}
where $w(u):=\mu_0(u)/p(u)$ is the \emph{importance weight}. Using $J$ samples $\{u^{(j)}\}_{j=1}^J$ which are distributed according to $p$, the IS estimator for $P_f$ is given by
\begin{align}
\hat{P}_f^{\mathrm{IS}} := \frac{1}{J}\sum_{j=1}^J I(G(u^{(j)})\le0)w(u^{(j)}).\label{IS estimator}
\end{align}
If the support of $p$ contains the intersection of the support of $\mu_0$ and the failure domain,~\eqref{IS estimator} gives an unbiased estimator for $P_f$. The optimal IS density is
\begin{align}
p_{\mathrm{opt}}(u) := \frac{1}{P_f}I(G(u)\le 0)\mu_0(u),\label{optimal IS density}
\end{align}
which leads to a zero-variance estimator. Sampling-based methods like SuS \cite{Au01, Au14} or SIS \cite{Papaioannou16, Wagner20} aim to generate samples from the optimal IS density. This viewpoint can be interpreted as seeking states $u\in\mathbb{R}^d$ which result in $G(u)\le 0$. It is possible to define this task as an inverse problem. In inverse problems, the goal is to identify the inputs of a model, whose outcome produces a set of given data $y^{\dagger}$. To define an equivalent inverse problem, we apply the ReLU function to the outcome of $G$ and define the auxiliary LSF $\widetilde{G}: \mathbb{R}^d\rightarrow\mathbb{R}$ as 
\begin{align}
\widetilde{G}(u) := \max \{0, G(u)\}.\label{tilde G}
\end{align}
With $\widetilde{G}$ and the data $y^{\dagger}=0$, we reformulate the rare event problem as an inverse problem. We seek all $u\in\mathbb{R}^d$ such that
\begin{align}
\widetilde{G}(u)=0. \label{rare event as inverse}
\end{align}
The overall goal of this manuscript is to apply the EnKF to generate samples from the failure domain and estimate the probability of failure with IS. Since inverse problems of the form~\eqref{rare event as inverse} are ill-posed, we discuss BIPs in the following section.

\begin{remark}
An alternative Bayesian interpretation of the rare event simulation problem is obtained through observation of equation~\eqref{optimal IS density}. The optimal IS density $p_{\mathrm{opt}}$ can be interpreted as a posterior density, where the indicator function $I(G(u)\le 0)$ is the likelihood function, $\mu_0$ is the prior density, and $P_f$ is the evidence. This observation is also discussed in~\cite{Uribe20}. However, this interpretation does not allow application of the EnKF algorithm, which is developed for standard BIPs with equality-type information. The latter implies that the model outcome is compared to the data with an equality sign. The likelihood $I(G(u)\le 0)$ provides inequality-type information.
\end{remark}

\subsection{Bayesian inverse problems}
We consider a forward model $\mathcal{G}:\mathbb{R}^d\rightarrow\mathbb{R}^m$, which maps the input to the output space. Given is the data $y^{\dagger}\in\mathbb{R}^m$. The goal in an inverse problem with noisy observations is to find a state $u^{\dagger}\in\mathbb{R}^d$ such that 
\begin{align}
y^{\dagger} = \mathcal{G}(u^{\dagger}) + \eta,\label{inverse problem}
\end{align}
where $\eta$ is the observational noise and is assumed to be distributed as $\mathrm{N}(0,\Gamma)$, where $\Gamma$ is a symmetric positive definite covariance matrix. Since the classical inverse problem~\eqref{inverse problem} is ill-posed, $(u,y)$ are modelled as realisations from a jointly varying random variable $(U,Y)$ \cite{Dashti2017, Stuart10}. With this Bayesian viewpoint and under mild assumptions on the forward model, noise distribution and prior distribution \cite{Latz20,Stuart10}, the inverse problem is well-posed and the solution is the posterior density $\mu_{y^{\dagger}}$. By virtue of Bayes' theorem \cite[Theorem 14]{Dashti2017}, the posterior density is given by
\begin{align*}
\mu_{y^{\dagger}}(u)=\frac{1}{Z(y^{\dagger})}\exp\left(-\Psi\left(u;y^{\dagger}\right)\right) \mu_0(u),
\end{align*}
where $\mu_0$ is the prior density and $\Psi$ is a potential, defined as $\Psi\left(u;y^{\dagger}\right):=1/2\Vert y^{\dagger}-\mathcal{G}(u)\Vert_{\Gamma}^2$, where $\Vert\cdot\Vert_{\Gamma}:=\Vert\Gamma^{-1/2}\cdot\Vert_2$. The term $\exp\left(-\Psi\left(u;y^{\dagger}\right)\right)$ is the likelihood function and returns the density of the data given a parameter state. The normalizing constant is given by
\begin{align*}
Z(y^{\dagger}):=\int_{\mathbb{R}^d}\exp\left(-\Psi\left(u;y^{\dagger}\right)\right)\mu_0(u)\mathrm{d}u>0.
\end{align*}
For linear models $\mathcal{G}(u)=Au$ with $A\in\mathbb{R}^{m\times d}$ and a Gaussian prior $\mu_0\sim\mathrm{N}(0,\Gamma_0)$, \cite[Theorem 2.4]{Stuart10} shows that the posterior mean $m\in\mathbb{R}^d$ and posterior covariance $C\in\mathbb{R}^{d\times d}$ are given by
\begin{align}
m = \left(A^T\Gamma^{-1}A+\Gamma_0^{-1}\right)^{-1}A^T\Gamma^{-1}y, \quad\quad C = \left(A^T\Gamma^{-1}A+\Gamma_0^{-1}\right)^{-1}.\label{posterior mean + covariance}
\end{align}
\\Similar to rare event estimation, sampling-based methods like Sequential Monte Carlo \cite{Moral06, Doucet11} have been developed to shift samples from the prior density $\mu_0$ to the posterior density $\mu_{y^{\dagger}}$. The EnKF for inverse problems \cite{Schillings16} is another sampling-based method.

\subsection{EnKF for inverse problems}
In this section, we use the notation and derivation of \cite{Schillings16}, which is motivated by Sequential Monte Carlo. In the EnKF for inverse problems, the posterior is reached in a sequential manner. Starting from the prior density $\mu_0$, the sequence of densities is defined by
\begin{align*}
\mu_n(u) \propto \exp\left(-\frac{nh}{2}\Vert y^{\dagger}- \mathcal{G}(u)\Vert_{\Gamma}^2\right)\mu_0(u),\quad\text{ for } n=1,\dots, N,
\end{align*}
where $h:=1/N$ and $N$ is the user-defined number of steps to reach the posterior density. It immediately follows that $\mu_N = \mu_{y^{\dagger}}$ is the posterior density. By the sequential definition, it holds that
\begin{align*}
\mu_{n+1}(u) \propto \exp\left(-\frac{h}{2}\Vert y^{\dagger}-\mathcal{G}(u)\Vert_{\Gamma}^2\right)\mu_n(u),\quad\text{ for } n=0,\dots, N-1.
\end{align*}
\begin{remark}\label{Remark downscaling}
The $n$-th step is equivalent to the solution of the inverse problem~\eqref{inverse problem} where the noise $\eta$ is distributed according to $\mathrm{N}(0,(nh)^{-1}\Gamma)$. If $nh>1$, i.e. $n>N$, the noise covariance is down-scaled.
\end{remark}
The sequence of densities $\mu_n$ is approximated via an ensemble of equally weighted particles. The initial ensemble $\mathbf{u_0}=\{u_0^{(j)}\}_{j=1}^J$ is distributed according to the prior density $\mu_0$. In one step of the EnKF, an ensemble of $J\in\mathbb{N}$ samples $\mathbf{u_n}=\{u_n^{(j)}\}_{j=1}^J$, which is distributed (approximately) according to the density $\mu_n$, is transformed into the ensemble $\mathbf{u_{n+1}} = \{u_{n+1}^{(j)}\}_{j=1}^J$, which is distributed (approximately) as $\mu_{n+1}$. Formally, the ensemble $\mathbf{u_n}$ is updated via
\begin{align}
u_{n+1}^{(j)} = u_{n}^{(j)} + C_{\mathrm{up}}(\mathbf{u_n})\left(C_{\mathrm{pp}}(\mathbf{u_n}) + \frac{1}{h}\Gamma\right)^{-1}\left(y_{n+1}^{(j)}-\mathcal{G}(u_{n}^{(j)})\right),\label{EnKF inverse uk}
\end{align}
for $j=1,\dots,J$, where $y_{n+1}^{(j)}$ is the data $y^{\dagger}$ perturbed by an additive Gaussian noise
\begin{align}
y_{n+1}^{(j)} = y^{\dagger} + \xi_{n+1}^{(j)},\label{nosie for y}
\end{align}
and $\xi_{n+1}^{(j)}\sim\mathrm{N}(0,h^{-1}\Gamma)$ represents the observational noise scaled by the step size $h^{-1}$. The matrices $C_{\mathrm{pp}}\in\mathbb{R}^{m\times m}$ and $C_{\mathrm{up}}\in\mathbb{R}^{d\times m}$ are the empirical covariance and cross-covariance matrices and are given by
\begin{align}
C_{\mathrm{pp}}(\mathbf{u_n}) &= \frac{1}{J}\sum_{j=1}^J \left(\mathcal{G}(u_{n}^{(j)})-\overline{\mathcal{G}}\right)\otimes\left(\mathcal{G}(u_{n}^{(j)})-\overline{\mathcal{G}}\right),\label{Cpp matrix}
\\ C_{\mathrm{up}}(\mathbf{u_n}) &= \frac{1}{J}\sum_{j=1}^J \left(u_{n}^{(j)}-\overline{u}\right)\otimes\left(\mathcal{G}(u_{n}^{(j)})-\overline{\mathcal{G}}\right),\label{Cup matrix}
\end{align}
where $\otimes$ denotes the outer product of two vectors and $\overline{u}$, $\overline{\mathcal{G}}$ are the empirical means
\begin{align*}
\overline{u} = \frac{1}{J}\sum_{j=1}^J u_{n}^{(j)}, \quad\quad \overline{\mathcal{G}} = \frac{1}{J}\sum_{j=1}^J \mathcal{G}(u_{n}^{(j)}).
\end{align*}
The authors of \cite{Schillings16} show that in the \emph{continuous time limit} $h\rightarrow 0$, the update~\eqref{EnKF inverse uk} leads to the following SDE
\begin{align}
\frac{\mathrm{d}u^{(j)}}{\mathrm{d}t} &= C_{\mathrm{up}}(\mathbf{u})\Gamma^{-1}\left(y^{\dagger}-\mathcal{G}(u^{(j)})\right) + C_{\mathrm{up}}(\mathbf{u})\Gamma^{-1/2}\frac{\mathrm{d}W^{(j)}}{\mathrm{d}t}\label{SDE}
\\&=\frac{1}{J}\sum_{k=1}^J\left\langle \mathcal{G}(u^{(k)})-\overline{\mathcal{G}}, y^{\dagger}-\mathcal{G}(u^{(j)}) + \Gamma^{1/2}\frac{\mathrm{d}W^{(j)}}{\mathrm{d}t}\right\rangle_{\Gamma}\left(u^{(k)}-\overline{u}\right),\notag
\end{align}
where $W^{(j)}$ is a Brownian motion \cite{Peres10}, which represents the observational noise. If we consider the noise free case, i.e. $y_{n+1}^{(j)} = y^{\dagger}$, and assume that the forward model $\mathcal{G}$ is linear, i.e. $\mathcal{G}(u)=Au$ for $A\in\mathbb{R}^{m\times d}$, it holds that
\begin{align}
\frac{\mathrm{d}u^{(j)}}{\mathrm{d}t} = -C_{\mathrm{uu}}(\mathbf{u})D_u\Psi(u^{(j)};y^{\dagger}) = - C_{uu}(\mathbf{u})\left(A^T\Gamma^{-1}Au^{(j)} - A^T\Gamma^{-1}y^{\dagger}\right),\label{EnKf inverse sample flow}
\end{align}
where $C_{uu}(\mathbf{u})$ is the empirical covariance matrix of the ensemble $\mathbf{u}$. Equation~\eqref{EnKf inverse sample flow} implies that the samples move in the direction of the negative gradient of the data misfit. The multiplication with $C_{\mathrm{uu}}(\mathbf{u})$ can be interpreted as a pre-conditioner of the particle dynamic \cite{Garbuno20}. In Section~\ref{Sec EnKF for affine linear LSFs}, we will show that a similar expression holds true for the setting of rare event estimation if we consider $\mathcal{G}=\widetilde{G}$ as defined in~\eqref{tilde G}.
\\The authors in \cite{Garbuno20} show for the linear case and the mean field limit $J\rightarrow\infty$ that the mean $m(t)$ and covariance $C(t)$ of the EnKF particles satisfy the following differential equations
\begin{align*}
\frac{\mathrm{d}}{\mathrm{d}t}m(t) &= -C(t)\left(A^T\Gamma^{-1}Am(t)-A^T\Gamma^{-1}y^{\dagger}\right),
\\\frac{\mathrm{d}}{\mathrm{d}t}C(t) &= -C(t)A^T\Gamma^{-1}AC(t),
\end{align*}
which lead to the solutions
\begin{align*}
m(t) = \left(A^T\Gamma^{-1}At+ \Gamma_0^{-1}\right)^{-1}A^T\Gamma^{-1}y^{\dagger},\quad\quad C(t) = \left(A^T\Gamma^{-1}At + \Gamma_0^{-1}\right)^{-1}.
\end{align*}
Thus for $t=1$, it holds that the mean and covariance of the EnKF ensemble is equal to the posterior mean and posterior covariance~\eqref{posterior mean + covariance}. However, for $t\rightarrow\infty$, the covariance $C(t)$ converges to zero, which gives ensemble collapse. Therefore, the EnKF for inverse problems gives samples of the posterior for $t=1$. For $t\rightarrow\infty$, the EnKF can be seen as an optimization method instead as a sampling method \cite{Garbuno20, Schillings16}. If $\mathcal{G}$ is nonlinear, the work of \cite{Ernst15} shows that the EnKF gives approximate samples from the analysis variable 
\begin{align}
U^a:= U+\mathrm{Cov}(U,\mathcal{G}(U) + \eta)\mathrm{Cov}(\mathcal{G}(U) + \eta)^{-1}(y^{\dagger} - \mathcal{G}(U) - \eta).\label{Ua}
\end{align}
We note that the distribution of $U^a$ is in general different from the posterior distribution. Moreover, $U^a$ represents a single EnKF update. Thus, for the continuous time limit and the limit $t\rightarrow\infty$, the analysis variable cannot be used to derive the limit of the mean $m(t)$, since the EnKF applies infinitely many update steps. 
\begin{remark}
The authors of \cite{Garbuno20} show that the EnKF can be adjusted in a way that the covariance of the ensemble does not collapse for $t\rightarrow \infty$. Instead, the posterior is reached for $t\rightarrow\infty$. In their approach, the noise is not added in the data space as in equation~\eqref{nosie for y}. Instead, it is added in the input space. However, for rare event estimation, the noise $\eta$ and noise covariance $\Gamma$ are artificial. Therefore, it is useful to consider $t>1$ for the EnKF approach in \cite{Schillings16}, which constitutes a downscaling of the noise as noted in Remark~\ref{Remark downscaling}. We will use a stopping criterion to determine when the downscaling of the noise is sufficient.
\end{remark}

\section{EnKF for rare event estimation}\label{Section EnKF}
In the following, we apply the EnKF approach of \cite{Schillings16} to the forward model $\mathcal{G}=\widetilde{G}$. With $\widetilde{G}$, the EnKF generates failure samples. To reduce the number of iterations, we apply the adaptive approach of \cite{Iglesias18}, which is similar to the adaptive approach of SIS given in \cite{Papaioannou16}. Since the EnKF does not provide samples from the posterior distribution but gives approximate samples from the analysis variable $U^a$, we apply an IS strategy to estimate the probability of failure. Additionally, we discuss an approach to handle multi-modal failure domains. 

\subsection{EnKF with adaptive step size}
We consider the case that the output space of the LSF $G$ is one-dimensional. This is the usual setting in rare event estimation and simplifies the following considerations. However, the EnKF for rare event estimation can be also applied to a general output space $\mathbb{R}^m$. By \cite{Kiureghian86, Hohenbichler81}, we assume, without loss of generality, that the underlying random variable $U:\Omega \rightarrow\mathbb{R}^d$ is a $d$-variate standard normally distributed random variable. Thus, we assume that the prior $\mu_0(u)=\varphi_d(u)$ is the $d$-variate standard Gaussian density. 
\\Translating the rare event setting into the Bayesian framework results in the data $y^{\dagger}=0$ and the auxiliary LSF $\widetilde{G}(u) = \max \{0, G(u)\}$, as mentioned in Section~\ref{Sec rare events}. Now, we can define the rare event problem as an inverse problem. The goal is to find $u^{\dagger}\in\mathbb{R}^d$ such that
\begin{align}
0 = \widetilde{G}(u^{\dagger}) + \eta,\label{rare event inverse}
\end{align}
where $\eta\sim\mathrm{N}(0,\Gamma)$ is observational noise and $\Gamma>0$. We note that the noise is artificial and we do not know $\Gamma$. However, it turns out that $\Gamma$ is not relevant for the EnKF with adaptive step size and the desired noise level is specified by the user. The posterior density of the inverse problem~\eqref{rare event inverse} is given by
\begin{align}
\mu_{y^{\dagger}}(u) \propto \exp\left(-\frac{1}{2\Gamma}\widetilde{G}(u)^2\right)\varphi_d(u).\label{posterior rare events}
\end{align} 
\begin{remark}
Since the auxiliary LSF $\widetilde{G}$ is always nonlinear, even if the original LSF is linear, the approximation of the derivative by the particles is not exact and the posterior is non-Gaussian. Thus by \cite{Ernst15}, the EnKF generates approximate samples from the analysis variable $U^a$~\eqref{Ua} which differs from the posterior density $\mu_{y^{\dagger}}$~\eqref{posterior rare events}.
\end{remark}
Compared with the SIS approach of \cite{Papaioannou16}, we can see that~\eqref{posterior rare events} is an alternative way to define a piecewise smooth approximation of the optimal IS density $p_{\mathrm{opt}}$ in~\eqref{optimal IS density}. In the same manner as in \cite{Iglesias18, Papaioannou16}, we determine the sequence of the EnKF densities adaptively. We define a sequence of temperatures $\sigma_n$ with $\infty=\sigma_0>\sigma_1>\sigma_2>\dots>\sigma_N>0$, where $N$ is the number of a priori unknown EnKF iterations. The sequence of densities in the EnKF is given by
\begin{align}
\mu_n(u) \propto \exp\left(-\frac{1}{2\sigma_n\Gamma}\widetilde{G}(u)^2\right)\varphi_d(u).\label{EnKF densities}
\end{align}
By the sequential definition of the EnKF it holds that
\begin{align*}
\mu_{n+1}(u) \propto \exp\left(-\frac{1}{2\Gamma}\left(\frac{1}{\sigma_{n+1}}-\frac{1}{\sigma_n}\right)\widetilde{G}(u)^2\right)\mu_{n}(u).
\end{align*}
By defining $t_n=1/\sigma_n$, the temperatures $\sigma_n$ can be viewed as a time discretization of the SDE~\eqref{SDE}. Considering the limit $\sigma_n\rightarrow 0$ in~\eqref{EnKF densities}, or $t_n\rightarrow\infty$, leads to pointwise convergence of the EnKF densities to the optimal IS density since
\begin{align*}
\lim_{\sigma_n\rightarrow 0} \exp\left(-\frac{1}{2\sigma_n\Gamma}\widetilde{G}(u)^2\right) = I(G(u)\le 0).
\end{align*} 
This behaviour is illustrated in Figure~\ref{indicator approximation}. We note that the limit $\sigma_n\rightarrow 0$ implies a complete downscaling of the noise $\Gamma$, which results in a classical ill-posed inverse problem. The approximation of the indicator function by SIS \cite{Papaioannou16} is given by
\begin{align*}
\lim_{\sigma_n\rightarrow 0} \Phi\left(-\frac{G(u)}{\sigma_n}\right) = I(G(u)\le 0), \quad\text{ for } G(u)\neq 0,
\end{align*}
where $\Phi(\cdot)$ is the cumulative distribution function of the one-dimensional standard normal distribution. Figure~\ref{indicator approximation} shows that the approximation by the EnKF is equal to the indicator function for $G(u)\le 0$, while the SIS approximation is symmetric around $G(u) = 0$.
\begin{figure}[htbp]
\centering
	\includegraphics[trim=0cm 1cm 0cm 0cm,scale=0.23]{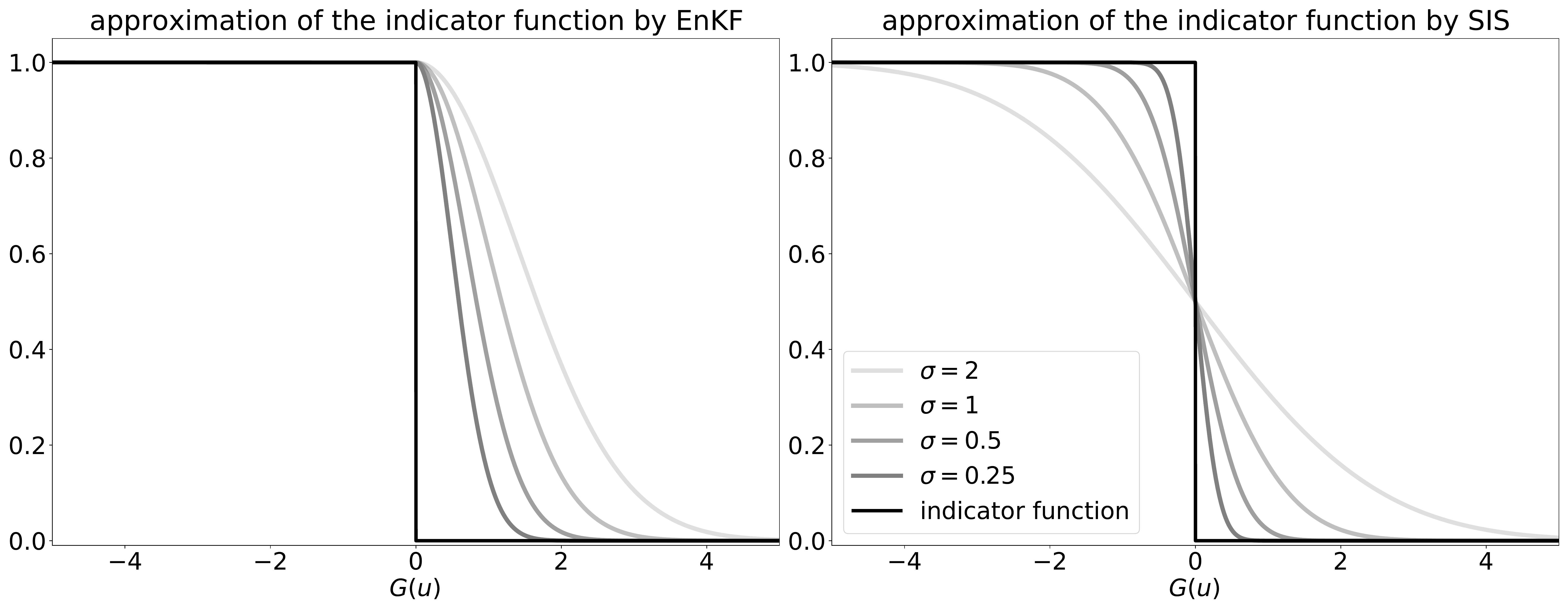}
		\caption[Indicator approximation]{Approximation of the indicator function $I(G(u)\le 0)$ by the EnKF and SIS.}
		\label{indicator approximation}
\end{figure}
\\Similar to \cite{Papaioannou16}, we use the coefficient of variation of the likelihood weights of two subsequent densities to determine the sequence $\sigma_n$. Therefore, we use a user specific target coefficient of variation $\delta_{\mathrm{target}}$. Given the temperature $\sigma_n$, we determine $\sigma_{n+1}$ by
\begin{align}
\sigma_{n+1} = \underset{\sigma \in (0,\sigma_n)}{\mathrm{argmin}}\quad \frac{1}{2}\left(\delta_{\mathbf{w_{n+1}}} - \delta_{\mathrm{target}}\right)^2,\label{Opimitzation sigma n+1}
\end{align} 
where $\delta_{\mathbf{w_{n+1}}}$ is the coefficient of variation of the likelihood weights $\mathbf{w_{n+1}} =\{w_{n+1}^{(j)}\}_{j=1}^J$ of the consecutive EnKF densities $\mu_n$ and $\mu_{n+1}$. These weights are given by
\begin{align*}
w_{n+1}^{(j)} = \exp\left(-\frac{1}{2\Gamma}\left(\frac{1}{\sigma_{n+1}}-\frac{1}{\sigma_{n}}\right)\widetilde{G}(u_n^{(j)})^2\right), \quad\text{ for } j=1,\dots, J.
\end{align*}
Through this adaptive procedure, it is not necessary to specify the variance $\Gamma$ of the observational noise. Therefore, we assume, without loss of generality, that $\Gamma = 1$. 
\\With the new temperature $\sigma_{n+1}$, we define the step size of the EnKF as $h_{n+1} = 1/\sigma_{n+1} - 1/\sigma_n$. Hence in one step of the EnKF, the ensemble $\mathbf{u_n}$ is updated by
\begin{align}
u_{n+1}^{(j)} = u_{n}^{(j)} + C_{\mathrm{up}}(\mathbf{u_n})\left(C_{\mathrm{pp}}(\mathbf{u_n}) + \frac{1}{h_{n+1}}\right)^{-1}\left(\xi_{n+1}^{(j)}-\widetilde{G}(u_{n}^{(j)})\right),\label{EnKF rare events uk}
\end{align}
for $j=1,\dots, J$, where $\xi_{n+1}^{(j)}\sim\mathrm{N}(0, h_{n+1}^{-1})$ is an additive scaled noise in the output space. The matrices $C_{\mathrm{pp}}(\mathbf{u_n})$ and $C_{\mathrm{up}}(\mathbf{u_n})$ are determined by~\eqref{Cpp matrix} and~\eqref{Cup matrix}, respectively, using $\widetilde{G}$ instead of $G$. Since the output space is one-dimensional, it holds that $C_{\mathrm{pp}}(\mathbf{u_n})\in\mathbb{R}$ and $C_{\mathrm{up}}(\mathbf{u_n})\in\mathbb{R}^{d}$. Hence,~\eqref{EnKF rare events uk} does not require the solution of a linear system and the computational costs of one EnKF step behaves as $\mathcal{O}(d\cdot J)$. 
\\As in SIS \cite{Papaioannou16}, we use the target coefficient of variation $\delta_{\mathrm{target}}$ as the stopping criterion of the EnKF iteration. In particular, the EnKF stops if the coefficient of variation of the likelihood weights with respect to the optimal IS density and the current EnKF density is less than $\delta_{\mathrm{target}}$. These likelihood weights are given by
\begin{align}
w_{\mathrm{opt}}^{(j)} = I\left(\widetilde{G}(u_n^{(j)})=0\right)\exp\left(\frac{1}{2\sigma_n}\widetilde{G}(u_{n}^{(j)})^2\right),\quad\text{ for } j=1,\dots,J.\label{weights optimal density}
\end{align}
The EnKF iteration stops if $\delta_{\textbf{w}_{\mathrm{opt}}} \le \delta_{\mathrm{target}}$. By the indicator function and $\widetilde{G}$, the weights $w_{\mathrm{opt}}^{(j)}$ are either zero or one and, therefore, the stopping criterion can be easily interpreted by the \emph{effective sample size} \cite[Section 3.4]{Latz18}. This implies that a certain percentage of the particles have to belong to the failure domain to stop the EnKF iteration. This percentage depends on $\delta_{\mathrm{target}}$. If $\delta_{\mathrm{target}}$ is small, then a high percentage of the particles has to belong to the failure domain. For instance, if $\delta_{\mathrm{target}}=0.25$, then around $94\%$ of the particles has to be in the failure domain. For $\delta_{\mathrm{target}}=3$, around $10\%$ of the particles has to be in the failure domain. We investigate different values for $\delta_{\mathrm{target}}$ in the numerical experiments.

\begin{remark}
We note that SIS applies an MCMC algorithm to shift samples between one density and its consecutive density. Therefore, the generated samples are distributed according to the target densities and the probability of failure can be directly estimated using these samples. However, the EnKF generates approximate samples from the analysis variable \cite{Ernst15}. Therefore, we do not use these samples to estimate the probability of failure directly. Instead, we use a single IS step after the EnKF iteration is finished, which yields an unbiased estimator for the probability of failure.
\end{remark}

\subsection{Estimation of the probability of failure}\label{Section estimate Pf}
Once the EnKF iteration is finished, the final ensemble is denoted by $\mathbf{u}_N$ and the final EnKF density is given by
\begin{align*}
\mu_N(u) \propto \exp\left(-\frac{1}{2\sigma_N}\widetilde{G}(u)^2\right)\mu_0(u).
\end{align*}
We follow \cite[Algorithm 3.1]{Kroese13} to construct an accurate and unbiased estimator for the probability of failure. Their approach is based on IS. Therefore, we fit a certain distribution model with the final ensemble $\textbf{u}_{N}$. Consequently, we generate the ensemble $\mathbf{\hat{u}} = \{\hat{u}^{(j)}\}_{j=1}^J$ from the fitted distribution and estimate the probability of failure by
\begin{align}
\hat{P}_f = \frac{1}{J}\sum_{j=1}^{J}I(G(\hat{u}^{(j)})\le 0)\frac{\varphi_d(\hat{u}^{(j)})}{p(\hat{u}^{(j)})},\label{Pf IS}
\end{align}
where $p(\cdot)$ is the pdf of the fitted distribution. In particular, we consider the GM \cite[Section 1]{Schnatter06} and the vMFNM distribution model \cite{Papaioannou19}.

The GM is defined by the weighted sum of $K$ Gaussian distributions, where each of them is defined by a mean vector $m_k$ and a covariance matrix $C_k$ yielding the pdf $\varphi(\cdot\mid m_k, C_k)$. Thus, the density of the GM distribution is given by
\begin{align}
p_{\mathrm{GM}}(u) = \sum_{k=1}^{K}\pi_k \varphi(u\mid m_k, C_k),\label{GM}
\end{align}
where $\pi_k\ge 0$ is the weight of the $k$th mixture component with $\sum_{k=1}^K \pi_k=1$. The parameters of the GM are determined by maximum likelihood estimation using the final EnKF ensemble $\mathbf{u_N}$. For $K>1$ mixture components, the parameters of the GM are not analytically given and the \emph{expectation-maximization} (EM) algorithm is applied. We refer to \cite{McLachlan05} for a detailed explanation of the EM algorithm.
\\The vMFNM distribution is based on the polar decomposition $u=r\cdot a$, where $r=\Vert u\Vert_2$ and $a=u/r$. For a single mixture component, the pdf of the vMFN distribution is given by
\begin{align*}
p_{\mathrm{vMFN}}(r,a\mid \nu, \kappa, s, \gamma) = p_{\mathrm{N}}(r\mid s,\gamma)\cdot p_{\mathrm{vMF}}(a\mid\nu,\kappa),
\end{align*}
where $p_{\mathrm{N}}$ is the density of the Nakagami distribution and $p_{\mathrm{vMF}}$ is the pdf of the von Mises--Fisher distribution. The Nakagami \cite{Nakagami60} distribution determines the distribution of the radius $r$ and depends on a shape parameter $s$ and spread parameter $\gamma$. The vMF distribution \cite{Wang16} determines the distribution of the direction $a$ and depends on the mean direction $\nu$ and concentration parameter $\kappa$. Similar to~\eqref{GM}, the vMFNM distribution is defined by
\begin{align*}
p_{\mathrm{vMFNM}}(r, a)=\sum_{k=1}^{K} \pi_k p_{\mathrm{vMFN}}(r,a\mid \nu_k, \kappa_k, s_k, \gamma_k).
\end{align*}
The parameters of the vMFNM distribution are determined by maximum likelihood estimation and the EM algorithm. Algorithm ~\ref{EnKF alg} shows the complete algorithm to estimate the probability of failure with the EnKF.
\begin{algorithm}
\caption{EnKF for rare event estimation}\label{EnKF alg}
\begin{algorithmic}[1]
\STATE generate the initial ensemble $\mathbf{u_0}$ by sampling $J$ independent samples from $\varphi_d(u)$
\STATE evaluate the auxiliary LSF $\widetilde{G}$ for the current ensemble $\mathbf{u_0}$
\STATE $\sigma_0 \leftarrow \infty$, $n\leftarrow 0$
\WHILE{EnKF is not finished}
	\STATE determine $\sigma_{n+1}$ from the optimization problem~\eqref{Opimitzation sigma n+1}
	\STATE define the current step size $h_{n+1}=1/\sigma_{n+1} - 1/\sigma_n$
	\STATE calculate $C_{\mathrm{pp}}(\mathbf{u_n})$ and $C_{\mathrm{up}}(\mathbf{u_n})$ based on~\eqref{Cpp matrix},\eqref{Cup matrix}
	\STATE generate $J$ independent samples $\xi_{n+1}^{(j)}\sim\mathrm{N}\left(0,h_{n+1}^{-1}\right)$
	\STATE update the ensemble $\mathbf{u_n}$ to $\mathbf{u_{n+1}}$ based on~\eqref{EnKF rare events uk}
	\STATE evaluate the auxiliary LSF $\widetilde{G}$ for the current ensemble $\mathbf{u_{n+1}}$
	\STATE determine the coefficient of variation $\delta_{\textbf{w}_{\mathrm{opt}}}$ based on~\eqref{weights optimal density}
	\IF{$\delta_{\textbf{w}_{\mathrm{opt}}}<\delta_{\mathrm{target}}$}
		\STATE EnKF is finished
	\ENDIF
	\STATE $n\leftarrow n+1$
\ENDWHILE
\STATE fit a distribution model based on the final ensemble $\mathbf{u_n}$
\STATE generate $J$ samples $\mathbf{\hat{u}}$ from the fitted distribution
\STATE estimate the probability of failure $\hat{P}_f$ based on~\eqref{Pf IS}
\RETURN $\hat{P}_f$
\end{algorithmic}
\end{algorithm}

\begin{remark}
We remark on the applicability of the density models to large-scale problems. The GM model is applicable for low-dimensional parameter spaces. For moderate- and high-dimensional spaces, the performance of the GM model deteriorates due to two reasons. Firstly, the number of parameters which have to be estimated for the GM model scales quadratically with respect to the input parameter dimension $d$~\cite{Papaioannou19}. Secondly, Gaussian distributions suffer from the concentration of measure phenomena. As discussed in~\cite{Katafygiotis08,Papaioannou19,Wang16}, the probability mass of the standard normal distribution in high dimensions is mostly concentrated near the \emph{important ring}, which is the hypersphere with radius $\sqrt{d}$. Hence, in high dimensions, it becomes challenging to fit a GM model such that it has significant probability content near the important ring. This leads to degenerated importance weights and most samples are associated with zero weights. The behavior of the GM model in high-dimensional problems is numerically demonstrated in \cite{Geyer19} in the context of the cross-entropy method.
\\The vMFNM distribution model extends the application of the EnKF to moderately high-dimensional parameter spaces. The number of parameters which have to be estimated for the vMFNM distribution model scales linearly with respect to $d$~\cite{Papaioannou19}. Moreover, the vMFNM distribution model operates in polar coordinates. Hence, for a proper choice of the radius parameters, all samples will be located in the proximity of the important ring, which avoids the weight degeneracy issue. We note that for high-dimensional parameter spaces, the number of samples per level $J$ has to be large to ensure that the parameters of the vMFNM distribution model are estimated accurately. A numerical study of the von Mises--Fisher distribution model with respect to an increasing input parameter dimension $d$ is given in~\cite{Wang16}.
\end{remark}
We have now defined the procedure of estimating the probability of failure via the EnKF. As we have seen, the GM and vMFNM distributions are able to capture mixture distributions. However, the EnKF in its current form is not able to generate samples from a mixture, since the samples are always concentrated around a single mean value.

\subsection{Multi-modal EnKF}\label{Subsection multi modal EnKF}
In this section, we investigate the approach of \cite{Reich19} such that the EnKF is able to generate samples from a multi-modal failure domain. This property is necessary for rare event estimation since it is possible that the failure domain consists of various distinct modes. 
\\To achieve the multi-modal property, the empirical covariance and cross-covariance matrices $C_{\mathrm{pp}}, C_{\mathrm{up}}$ are localised for each particle of the ensemble. Therefore, we calculate a weight matrix $W\in\mathbb{R}^{J\times J}$, which represents the distances of the particles. The entries of the weight matrix are given by 
\begin{align}
W_{i,j} = \exp\left(-\frac{1}{2\alpha}\Vert u_n^{(i)} - u_n^{(j)}\Vert_2^2\right), \quad\text{ for } i,j=1,\dots,J,\label{weight matrix}
\end{align}
where $\alpha>0$ is a parameter chosen by the user. In addition, the weight matrix is normalised such that each column sums up to one. A large value of $\alpha$ leads to higher weights for the neighbouring particles. For small $\alpha$, the weights are only large for the nearest neighbours, which yields more localised covariance matrices and the particles move slowly since the particles do not interact with their neighbours \cite{Reich19}.
\begin{remark}\label{Remark adaptive approach multi modal}
We note that~\eqref{weight matrix} is the value of the pdf $\varphi_d(u_n^{(i)}\mid u_n^{(j)}, \alpha\cdot\mathrm{Id}_d)$, i.e., the Gaussian density with mean vector $u_n^{(j)}$ and covariance matrix $\alpha\cdot \mathrm{Id}_d$, where $\mathrm{Id}_d\in\mathrm{R}^{d\times d}$ denotes the identity matrix. 
\end{remark}
To avoid the selection of the parameter $\alpha$, we propose the following adaptive approach. We first apply a distribution-based clustering through fitting a mixture distribution model. Either of the two distribution models discussed in Section~\ref{Section estimate Pf} can be used for this purpose. The ensemble $\mathbf{u_n}$ is thus split into $K$ clusters. For each cluster $k=1,\dots,K$, we determine the empirical covariance matrix $C_k$ by using all particles which belong to the cluster $k$. Finally, for the particle $u_n^{(j)}$ which belongs to the cluster $k$, the $j$th column of the weight matrix is determined by
\begin{align}
W_{i,j} = \exp\left(-\frac{1}{2}\Vert C_k^{-1/2}(u_n^{(i)} - u_n^{(j)})\Vert_2^2\right), \quad\text{ for } i=1,\dots,J.\label{weight matrix adaptive}
\end{align}
\begin{remark}\label{Not adaptive approach multi modal}
Alternative to this adaptive procedure, the parameter $\alpha$ in~\eqref{weight matrix} can be chosen as $\alpha\propto d$. This approach could be applied in high-dimensional problems to address the curse of dimensionality of the Euclidean norm. This is due to the fact that the Euclidean distance of the points $u_1 = (0,\dots,0)\in\mathbb{R}^d$ and $u_2=(\epsilon,\dots,\epsilon)^T\in\mathbb{R}^d$ satisfies
\begin{align*}
\Vert u_1 - u_2\Vert_2^2 = d\epsilon^2.
\end{align*}
\end{remark}
With the weight matrix $W$, we calculate for each particle localised means and covariance matrices. The localised means for the particle $u_n^{(j)}$ are given by
\begin{align*}
\overline{u}_{\mathrm{loc}}^{(j)} = \sum_{i=1}^J W_{i,j} u_n^{(i)}, \quad \quad \overline{\widetilde{G}}_{\mathrm{loc}}^{(j)} = \sum_{i=1}^J W_{i,j} \widetilde{G}(u_n^{(i)}).
\end{align*}
With these localised means, we determine the localised covariance matrices by
\begin{align}
C_{\mathrm{loc}, \mathrm{pp}}(u_n^{(j)}) &= \sum_{i=1}^J W_{i,j} \left(\widetilde{G}(u_{n}^{(i)})-\overline{\widetilde{G}}_{\mathrm{loc}}^{(j)}\right)\otimes\left(\widetilde{G}(u_{n}^{(i)})-\overline{\widetilde{G}}_{\mathrm{loc}}^{(j)}\right),\label{C_pp local}
\\ C_{\mathrm{loc}, \mathrm{up}}(u_n^{(j)}) &= \sum_{i=1}^J W_{i,j}\left(u_{n}^{(i)}-\overline{u}_{\mathrm{loc}}^{(j)}\right)\otimes\left(\widetilde{G}(u_{n}^{(i)})-\overline{\widetilde{G}}_{\mathrm{loc}}^{(j)}\right).\label{C_up local}
\end{align}
Finally, one iteration of the EnKF with localised covariance matrices is given by 
\begin{align}
u_{n+1}^{(j)} = u_{n}^{(j)} + C_{\mathrm{loc},\mathrm{up}}(u_{n}^{(j)})\left(C_{\mathrm{loc}, \mathrm{pp}}(u_{n}^{(j)}) + \frac{1}{h_{n+1}}\right)^{-1}\left(\xi_{n+1}^{(j)}-G(u_{n}^{(j)})\right),\label{EnKF update multi modal}
\end{align}
for $j=1,\dots,J$. After the EnKF iteration is finished, we fit a mixture distribution model and estimate the probability of failure as in Section~\ref{Section estimate Pf}. We summarise the EnKF update with localised covariances in Algorithm~\ref{EnKF alg multi modal}.

\begin{algorithm}
\caption{EnKF update with localised covariances}\label{EnKF alg multi modal}
\begin{algorithmic}[1]
\IF{$\alpha$ is given}
	\FOR{all particles $u_n^{(j)}$}
		\STATE determine the $j$th column of the normalized weight matrix by~\eqref{weight matrix}
	\ENDFOR 
\ELSE
	\STATE split the ensemble $\mathbf{u_n}$ into $K$ clusters
	\FOR{all clusters $k$}
		\STATE determine all particles which belong to the cluster $k$
		\STATE determine the empirical covariance matrix $C_k$
		\FOR{all particles $u_n^{(j)}$ belonging to the cluster $k$}
			\STATE determine the $j$th column of the normalized weight matrix by~\eqref{weight matrix adaptive}
		\ENDFOR
	\ENDFOR
\ENDIF
\FOR{all particles $u_n^{(j)}$}
	\STATE determine $C_{\mathrm{loc}, \mathrm{pp}}(u_n^{(j)})$, $C_{\mathrm{loc}, \mathrm{up}}(u_n^{(j)})$ by~\eqref{C_pp local}, ~\eqref{C_up local}
	\STATE update $u_n^{(j)}$ to $u_n^{(j+1)}$ by~\eqref{EnKF update multi modal}
\ENDFOR
\end{algorithmic}
\end{algorithm}

\section{EnKF for affine linear LSFs: theoretical properties}\label{Sec EnKF for affine linear LSFs}
In this section, we derive theoretical properties of the EnKF for rare event estimation. We consider the case where the LSF $G$ is affine linear, i.e. $G(u)=a^Tu-b$, where $a\in\mathbb{R}^d$ and $b\in\mathbb{R}$. The linearity of $G$ implies that the auxiliary LSF $\widetilde{G}=\max \{0, G(u)\}$ is piecewise linear. 
\\In particular, we derive the continuous time limit $h\rightarrow 0$ of the EnKF update~\eqref{EnKF rare events uk} and investigate the large particle limit $J\rightarrow\infty$ of the resulting SDE. Thereafter, we study the large time properties $t\rightarrow\infty$ to determine the limit of the ensemble mean. We will always consider the noise free case, i.e., $y_{n+1}^{(j)}=y^{\dagger}$. We do not analyse the analysis variable in~\eqref{Ua}, since we are particularly interested in the continuous time limit and the large time properties.
\\In the next theorem, we derive the continuous time limit $h\rightarrow 0$ of the EnKF update~\eqref{EnKF rare events uk}. We note that we add the failure surface $\{G=0\}$ to the safe domain for the remaining section. This has no influence on the probability of failure estimate since the failure surface is a set with Lebesgue measure equal to zero. However, adding $\{G=0\}$ to the safe domain simplifies the following considerations since safe particles remain safe states and failure particles remain failure states for all $t\ge 0$.
\begin{theorem}\label{Theorem linear noise free}
Denote by $S=\{k\in\lbrack J \rbrack :G(u^{(k)})\ge0\}$ and $F=\{k\in\lbrack J\rbrack: G(u^{(k)})<0\}$ the index sets of the safe and failure particles. For the LSF $G(u)=a^Tu-b$ and for the noise free case $y_{n+1}^{(j)}=y^{\dagger}$, the safe particles satisfy the flow 
\begin{align}
\frac{\mathrm{d}u^{(j)}}{\mathrm{d}t} = -C_{\mathrm{uu}}(\mathbf{u})D_u\left(\frac{1}{2}G(u^{(j)})^2\right) + \frac{G(u^{(j)})}{J}\sum_{k\in F}G(u^{(k)})(u^{(k)}-\overline{u}),\label{gradient flow rare events}
\end{align}
while failure particles do not move.
\end{theorem}
\begin{remark}
The first summand of~\eqref{gradient flow rare events} implies that the safe particles move in the direction where the LSF $G$ decreases. If the mean $\overline{u}$ is in the safe domain, it follows that $G(u^{(k)})(u^{(k)}-\overline{u})$ points away from the failure surface since $G(u^{(k)})<0$ for $k\in F$. Thus, the second term slows down the movement to the failure surface. Indeed, we prove that the safe particles converge to the surface of the failure domain.
\end{remark}
From Theorem~\ref{Theorem linear noise free}, we observe that initial failure particles do not move in the EnKF iterations. If initial safe particles reach the failure surface, they will stay at the failure surface for the remaining time. Thus, initial safe particles follow the dynamic~\eqref{gradient flow rare events} and do not contribute to the second term in~\eqref{gradient flow rare events} for all $t\ge 0$. If the initial ensemble does not contain failure particles, the second term in~\eqref{gradient flow rare events} is zero for all $t\ge 0$ and the dynamic simplifies. Therefore, we distinguish the cases if initial failure particles are present or not.
\\In the following, we consider the continuous time limit~\eqref{gradient flow rare events} for an infinite ensemble $J\rightarrow\infty$. We denote by $U(t)$ the random variable which is distributed according to the particle density of the EnKF particles $\{u^{(j)}\}_{j=1}^{\infty}$ at the time point $t\ge 0$. Moreover, we denote by $m(t):=\mathbb{E}[U(t)]$ the mean of the ensemble and by $C(t):=\mathbb{E}[(U(t)-m(t))\otimes(U(t)-m(t))]$ the covariance matrix of the ensemble. We derive the limit of the ensemble mean under the following assumptions. 
\begin{assumption}\label{Ass lemma}
We assume that
\begin{itemize}
\item[(i)] $G(u) = a^Tu-b$ with $a=\left(1,0,\dots,0\right)^T\in\mathbb{R}^d$ and $b<0$,
\item[(ii)] the distribution of the input random variable $U$ is $d$-variate independent standard normal,
\item[(iii)] the EnKF is applied without noise, i.e., $y_{n+1}^{(j)}=y^{\dagger}$. 
\end{itemize}
\end{assumption}
We note that Assumption~\ref{Ass lemma}~(i) can be satisfied without loss of generality by rotation invariance of the standard normal density. 

\subsection{Mean-field limit: no failure particles}\label{Sec no failure}
We begin with the case that the initial ensemble does not contain failure particles. Thus, the second term in~\eqref{gradient flow rare events} is zero and the distinction of safe and failure states is not necessary. We note that this assumption is not valid, if we consider the large particle limit $J\rightarrow\infty$ and a Gaussian initial ensemble, since the initial ensemble will contain failure states. However, this assumption simplifies the proofs of the statements and it enables us to obtain insights in the particle dynamic and mean-field limit. We consider the case that the initial ensemble contains failure samples in Section~\ref{Sec with failure samples}.
\\In the next lemma, we show that the ensemble mean $m(t)$ converges to the MLFP $u^{\mathrm{MLFP}}\in\mathbb{R}^d$ under Assumption~\ref{Ass lemma} with the restriction that the initial ensemble contains no failure particles. The MLFP is the solution of the minimization problem
\begin{align*}
\underset{u\in\mathbb{R}^d}{\mathrm{min}} \quad \frac{1}{2} \Vert u \Vert_2^2, \quad\text{ such that } G(u) = 0.
\end{align*}
For general nonlinear problems, the FORM approximation of $P_f$ is given by $P_f^{\mathrm{FORM}}=\Phi\left(-\Vert u^{\mathrm{MLFP}}\Vert_2\right)$ if $G(0)>0$. If $G(0)<0$, $P_f^{\mathrm{FORM}}= 1 - \Phi\left(-\Vert u^{\mathrm{MLFP}}\Vert_2\right)$. Indeed, $P_f^{\mathrm{FORM}}$ is equal to $P_f$ if the LSF $G$ is affine linear. Assumption~\ref{Ass lemma}~(i) implies that $u^{\mathrm{MLFP}}=\left(b,0,\dots,0\right)^T$, $G(0)>0$, and $\Vert u^{\mathrm{MLFP}} \Vert_2 = \vert b \vert$. 

\begin{theorem}\label{Lemma FORM}
Let Assumption~\ref{Ass lemma} hold under the restriction that the initial ensemble contains no failure particles even for an infinite ensemble. In the large particle limit $J\rightarrow\infty$, the ensemble mean satisfies
\begin{align*}
m_1(t) = b\left(1 - \frac{1}{\sqrt{2t+1}}\right), \quad\quad m_i(t)=0, \quad\text{ for } i=2,\dots,d,
\end{align*}
while the covariance of the ensemble satisfies 
\begin{align}
C(t) = \begin{pmatrix}
1/(1+2t) & \\
 & \mathrm{Id}_{d-1}
\end{pmatrix}.\label{equation covariance thm}
\end{align}
Thus, for the limit $t\rightarrow\infty$, it holds
\begin{align*}
\lim_{t\rightarrow\infty} m(t) = u^{\mathrm{MLFP}}.
\end{align*}
\end{theorem}
\begin{corollary}
Let Assumption~\ref{Ass lemma} hold under the restriction that the initial ensemble contains no failure particles even for an infinite ensemble. The relative distance between $u^{\mathrm{MLFP}}$ and the ensemble mean $m(t)$ satisfies
\begin{align*}
\frac{\Vert m(t) - u^{\mathrm{MLFP}}\Vert_2}{\Vert u^{\mathrm{MLFP}}\Vert_2} = \frac{1}{\sqrt{2t+1}}.
\end{align*}
\end{corollary}
\begin{remark}
From equation~\eqref{equation covariance thm}, we see that the covariance $C(t)$ implies ensemble collapse in the first component for $t\rightarrow\infty$. However, the ensemble does not collapse to a single point but it collapses to the surface of the failure domain. Thus, we see that the particles move only in direction perpendicular to the failure surface, or equivalently, in direction of $a$, until all particles are on the failure surface.
\end{remark}

\subsection{Mean-field limit: with failure particles}\label{Sec with failure samples}
In this section, we consider the case that the initial ensemble contains failure particles. To derive the mean-field equation, we split the ensemble into the safe and failure particles. We define $U_S(t)$ as the random variable which is distributed according to the particle density of the safe particles $\{u^{(j)} : G(u^{(j)})\ge0\}_{j=1}^{\infty}$ at $t\ge 0$. Similar, $U_F(t)$ is the random variable which is distributed according to the particle density of the failure particles $\{u^{(j)} : G(u^{(j)})<0\}_{j=1}^{\infty}$ at $t\ge 0$. We note that the distribution of $U_F(t)$ stays constant with respect to $t$ since failure particles do not move. 
\\Since the portion of failure particles in the initial ensemble is equal to $P_f$ for $J\rightarrow\infty$, it holds that $\mathbb{P}(U(t) = U_S(t)) = (1-P_f)$ and $\mathbb{P}(U(t)=U_F(t)) = P_f$. Thus, the mean of the ensemble is given by
\begin{align}
m(t)= \mathbb{E}[U(t)] &= (1-P_f)\mathbb{E}[U_S(t)] + P_f\mathbb{E}[U_F(t)]\notag
\\ &=:(1-P_f)m_S(t)+P_fm_F(t).\label{mean safe and failure}
\end{align}
The mean $m_F(t)$ of the failure particles is constant and is equal to the mean of the optimal IS density, which is given by
\begin{align}
m_F = \mathbb{E}[U\mid G(U)<0] = \mathbb{E}[U\mid U_1<b] = \left(- \frac{\varphi_1(b)}{\Phi(b)},0,\dots,0\right)^T =:u^{\mathrm{opt}},\label{mean opt}
\end{align}
where we apply the formula of the mean of a truncated Gaussian \cite[Section 10.1]{Johnson94}. Indeed, $u^{\mathrm{opt}}$ is the mean of a standard Gaussian random variable truncated at $U_1 < b$. The following theorem states that the ensemble mean $m(t)$ converges to a convex combination of the MLFP $u^{\mathrm{MLFP}}$ and the mean of the optimal IS density $u^{\mathrm{opt}}$ under Assumption~\ref{Ass lemma}.

\begin{theorem}\label{Thm mean failure}
Let Assumption~\ref{Ass lemma} hold. For the large particle limit $J\rightarrow\infty$, the ensemble mean satisfies
\begin{align}
\lim_{t\rightarrow\infty} m(t) = (1-P_f) u^{\mathrm{MLFP}} + P_f u^{\mathrm{opt}}.\label{eq thm mean failure}
\end{align}
\end{theorem}
In the following, we point out the structure of the proof of Theorem~\ref{Thm mean failure}, since this gives important insights to the dynamic of the EnKF particles. The formal proofs are given in the Appendix~\ref{Appendix}.
From~\eqref{mean safe and failure} and~\eqref{mean opt}, it is sufficient to show that the mean of the safe particles $m_S(t)$ converges to $u^{\mathrm{MLFP}}$. To prove this statement, we derive the mean field equation of the particle dynamic~\eqref{gradient flow rare events}.
\begin{lemma}\label{Lemma mean field limit}
Let Assumption~\ref{Ass lemma} hold. For the large particle limit $J\rightarrow\infty$, the mean field equation of the safe particles is given by
\begin{align}
\frac{\mathrm{d}u_S(t)}{\mathrm{d}t} = -C(t)D_u\left(\frac{1}{2}G(u_S(t))^2\right) + G(u_S(t)) P_f\left((1,0,\dots,0)^T -m(t)(u_1^{\mathrm{opt}} - b)\right).\label{mean field limit failure}
\end{align}
\end{lemma}

\begin{remark}\label{Remark 1}
Since $D_u\left(\frac{1}{2}G(u)^2\right) = \left(G(u),0,\dots,0\right)^T$, $C(0)=\mathrm{Id}_d$ and $m(0)=0$, we see from~\eqref{mean field limit failure} that the dynamic acts only on the first component $u_{S,1}(t)$ of the particles for $t=0$. Moreover, this movement is independent of the other components $u_{S,i}(t)$ for $i=2,\dots,d$. Hence, the covariance matrix $C(t)$ and the mean $m(t)$ will only change in their fist components $C_{1,1}(t)$ and $m_1(t)$. Inductively, we conclude that the particle dynamic~\eqref{mean field limit failure} acts only on the first component $u_{S,1}(t)$ for all $t\ge 0$ while all other components $u_{S,i}(t)$ for $i=2,\dots,d$ remain constant. We note that this observation is also shown in~\eqref{equation covariance thm} for the case that no failure particles are present.
\end{remark}
The above remark implies that it is sufficient to consider the case $d=1$. In this case, the mean field equation of the safe particles reads as
\begin{align*}
\frac{\mathrm{d}u_S(t)}{\mathrm{d}t} = -C(t)G(u_S(t)) + G(u_S(t)) P_f\left(1 -m(t)(u^{\mathrm{opt}} - b)\right),
\end{align*}
where $C(t)\in\mathbb{R}$ is the variance of the ensemble. Thus, the mean $m_S(t)$ of the safe particles satisfies the flow 
\begin{align}
\frac{\mathrm{d}m_S(t)}{\mathrm{d}t} = -G(m_S(t))\left(C(t) - P_f\left(1 -m(t)(u^{\mathrm{opt}} - b)\right)\right).\label{dynamic safe mean}
\end{align}
To prove Theorem~\ref{Thm mean failure}, we show that the mean of the safe particles converges to $u^{\mathrm{MLFP}}$. Since $m_S(t)\ge u^{\mathrm{MLFP}}$ and $G(m_S(t))\ge0$, it is sufficient to show that
\begin{align}
C(t) - P_f\left(1 -m(t)(u^{\mathrm{opt}} - b)\right)>0, \quad\text{ for all } t\ge 0,\label{covariance ineq}
\end{align}
as long as $m_S(t)\neq u^{\mathrm{MLFP}}$. This guarantees that the dynamic~\eqref{dynamic safe mean} is negative and, thus, $m_S$ converges to $u^{\mathrm{MLFP}}$. Since the ensemble contains always the initial failure particles, we can bound the variance $C(t)$ from below.

\begin{lemma}\label{lemma covariance}
We consider $d=1$. If $m_S(t)\neq u^{\mathrm{MLFP}}$ and the safe particles are not collapsed to a single point, the variance $C(t)$ is bounded from below by
\begin{align*}
C(t) > (1-P_f)m_S(t)^2 + P_f(1+bu^{\mathrm{opt}})-((1-P_f)m_S(t)+P_f u^{\mathrm{opt}})^2.
\end{align*}
\end{lemma}
In the proof of Theorem~\ref{Thm mean failure}, see Appendix~\ref{Appendix sec with failure}, we show with Lemma~\ref{lemma covariance} that~\eqref{covariance ineq} is valid for all $t\ge 0$ as long as $m_S(t)\neq u^{\mathrm{MLFP}}$. Together with the fact that $\mathrm{d}m_S(t)/\mathrm{d}t =0$ for $m_S(t) = u^{\mathrm{MLFP}}$, we conclude that
\begin{align*}
\lim_{t\rightarrow\infty} m_S(t) = u^{\mathrm{MLFP}},
\end{align*}
which implies that~\eqref{eq thm mean failure} holds true.
\begin{remark}
We note that the derived theoretical properties only hold true if the EnKF is applied without noise. In this case, the safe particles converge to the surface of the failure domain and remain there. However, the EnKF is more robust if noise is added to the observations. Therefore, we propose to employ the EnKF with noise for practical applications.
\end{remark}

\begin{remark}
Theorem~\ref{Lemma FORM} and~\ref{Thm mean failure} give a justification to consider the limit $t\rightarrow\infty$ or $\sigma\rightarrow 0$ even for nonlinear LSFs since we expect that the particles converge to the failure domain and a high number of particles is in proximity of the MLFP. 
\end{remark}

\begin{remark}
Since $P_f$ is typically small, we see from Theorem~\ref{Thm mean failure} that the EnKF can be applied to estimate $u^{\mathrm{MLFP}}$ and to approximate $P_f$ by $P_f^{\mathrm{FORM}}$.
\end{remark}

\section{Numerical experiments}\label{chapter numerical experiments}
We consider four numerical experiments to test the performance of the EnKF for rare event estimation. In all experiments, we compare the results of the EnKF with SIS. We consider the SIS algorithm given in \cite{Wagner20} which applies the vMFNM distribution as proposal density in the MCMC algorithm. In all experiments, we apply SIS without burn-in and $10\%$ of the samples are chosen as seeds for the simulated Markov chains via multinomial resampling. Hence, the simulated Markov chains have length equal to $10$. The EnKF is always applied with noise $\xi_{n+1}^{(j)}\sim\mathrm{N}(0, h_{n+1}^{-1})$, where $h_{n+1}$ is the step size of the EnKF update. We refer to the standard EnKF iteration~\eqref{EnKF rare events uk} as the \emph{EnKF with global covariances} and the multi-modal approach of Section~\ref{Subsection multi modal EnKF} as the \emph{EnKF with local covariances}.
\\The performance is measured via the required computational costs and the achieved \emph{relative root mean square error} (relRMSE), which is defined as
\begin{align*}
\mathrm{relRMSE} := \frac{\left(\mathbb{E}\Bigl\lbrack(\hat{P}_f-P_f)^2\Bigr\rbrack\right)^{\frac{1}{2}}}{P_f},
\end{align*}
where $\hat{P}_f$ is the estimated probability of failure by the EnKF or SIS. The computational costs are equal to the number of required LSF evaluations. In all plots, we remove the probability of failure estimates which are larger than the $99$th percentile of the estimates. This removes outliers\footnote[1]{An estimate $\hat{x}$ is an outlier if $\hat{x}\ge x_{0.75}+3(x_{0.75}-x_{0.25})$, where $x_{0.25}$ and $x_{0.75}$ are the $25$th and $75$th percentiles. This criterion is based on Tukey's fences \cite[Section 2D]{Tukey77} and indicates that an estimate is \emph{far out}.} which might occur during the iterations. We are aware that these occur due to the choice of the parametric model of the IS density $p$ which leads to large likelihood weights in~\eqref{Pf IS} in very infrequent simulation runs. This observation in also discussed in~\cite[Section 9.3]{Owen13}. The percentage of outliers in the EnKF estimates is between $0.6\%-2.0\%$ for the first three examples and $3.2\%$ for the fourth example. Hence, by removing particles which are larger than the $99$th percentile, not all outliers are removed. The percentage of outliers in the SIS estimates is smaller than $1\%$.
\\We begin with three examples which are also considered in \cite{Papaioannou16}. These examples do not require expensive forward model evaluations and their inputs are two-dimensional independent standard normal random variables. Therefore, we can visualize how the ensemble of particles is evolving during the EnKF iterations and compare them with the samples generated by SIS. In the fourth example, we consider the diffusion equation in one-dimensional space with stochastic diffusion coefficient. This example is also considered in \cite{Ullmann15, Wagner20} and has a high-dimensional parameter space. The $\mathrm{relRMSE}$ and the average computational cost are estimated with $500$ independent simulation runs for the first three examples and with $100$ runs for the fourth example.

\subsection{Convex limit-state function}
We consider the following convex LSF, which is given in \cite{katsuki1994hyperspace},
\begin{align*}
G(u) = 0.1(u_1-u_2)^2 - \frac{1}{\sqrt{2}}(u_1+u_2) + 2.5.
\end{align*}
The corresponding probability of failure is $4.21 \cdot 10^{-3}$ and the failure domain has a single mode \cite{Papaioannou16}. We consider $J\in\{250, 500, 1000, 2000, 5000, 10000\}$ as the ensemble sizes and $\delta_{\mathrm{target}}\in\{0.25, 0.50, 1.00, 2.00, 5.00, 10.00\}$ as the target coefficients of variation. Since the failure domain is unimodal, we apply the EnKF with global covariances. We apply either the GM or vMFNM distribution model with one mixture component to fit the distribution of the final EnKF particles. We apply SIS with the vMFNM distribution model with one mixture component as the proposal density in the MCMC step.
\\Figure~\ref{Figure: example 1 samples} shows the samples of the final iteration of the EnKF and SIS for the convex LSF. Note that these are not the samples of the fitted distribution after the final step of the EnKF. We see that a small target coefficient of variation leads to more samples within the failure domain. This holds true for both, the EnKF and SIS. We observe that the EnKF samples are more spread along the surface of the failure domain. By Theorem~\ref{Lemma FORM} this observation is expected since the covariance of the ensemble stays constant in the direction parallel to the failure surface. In addition, the mean of the EnKF particles is in proximity to the MLFP and the mean of the optimal IS density as expected by Theorem~\ref{Thm mean failure}. In contrast, the SIS samples are centered around a certain mean value and are more similar to the optimal IS density. For $\delta_{\mathrm{target}} = 1.00$, the EnKF moves more samples into the failure domain than SIS. Therefore, we expect that for larger $\delta_{\mathrm{target}}$ the EnKF performs better than SIS.
\begin{figure}[htbp]
\centering
	\includegraphics[trim=0cm 1cm 0cm 0cm,scale=0.23]{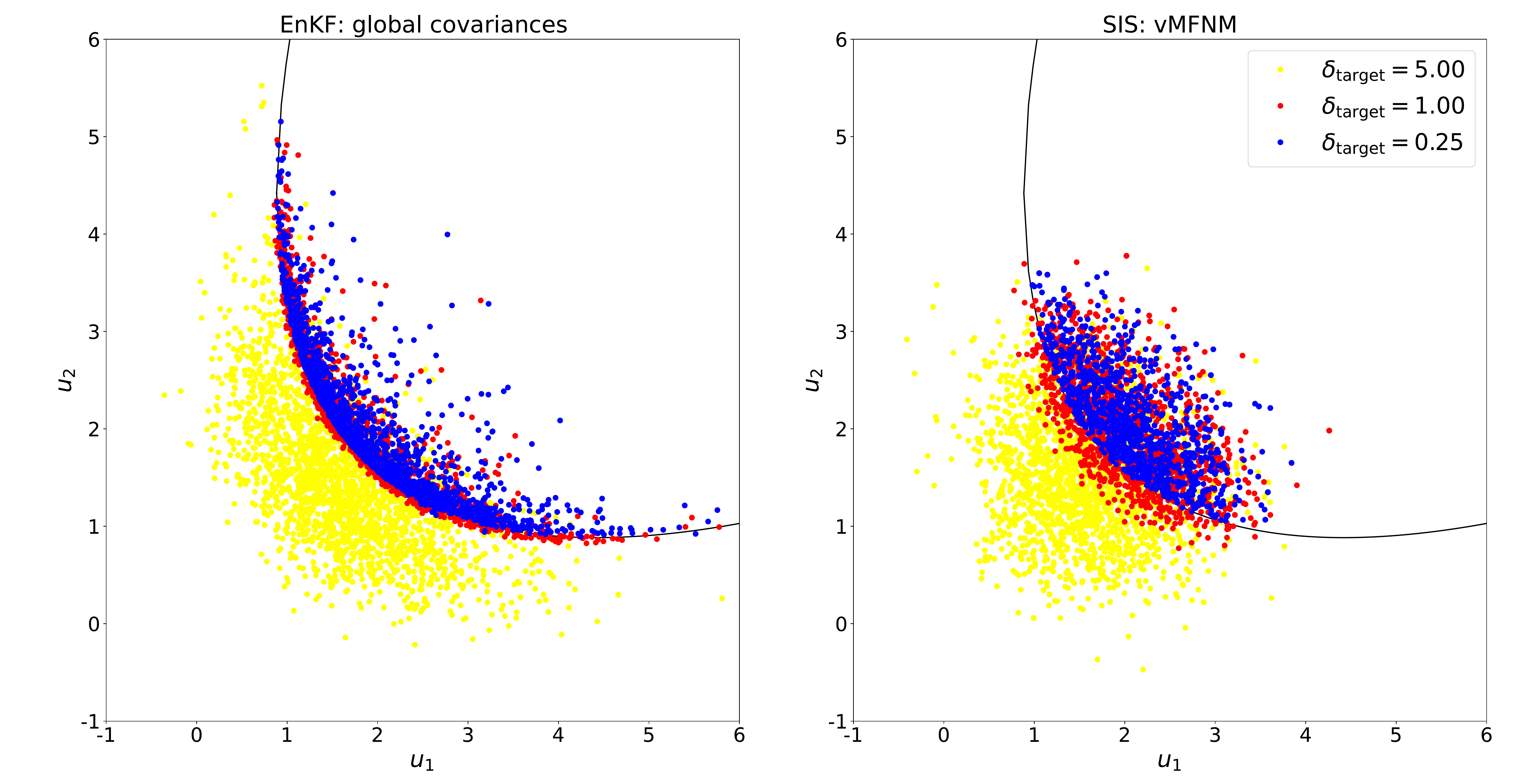}
	\caption{Convex LSF: Samples at the end of the iterations of the EnKF and SIS for $\delta_{\mathrm{target}}\in\{0.25, 1.00, 5.00\}$ and for $2000$ samples per level. The black lines show the boundary of the failure domain. Left: Samples of the EnKF with global covariances. Right: Samples of SIS with vMFNM and one mixture component.}\label{Figure: example 1 samples}
\end{figure}
\begin{figure}[htbp]
\centering
	\includegraphics[trim=0cm 1cm 0cm 0cm,scale=0.23]{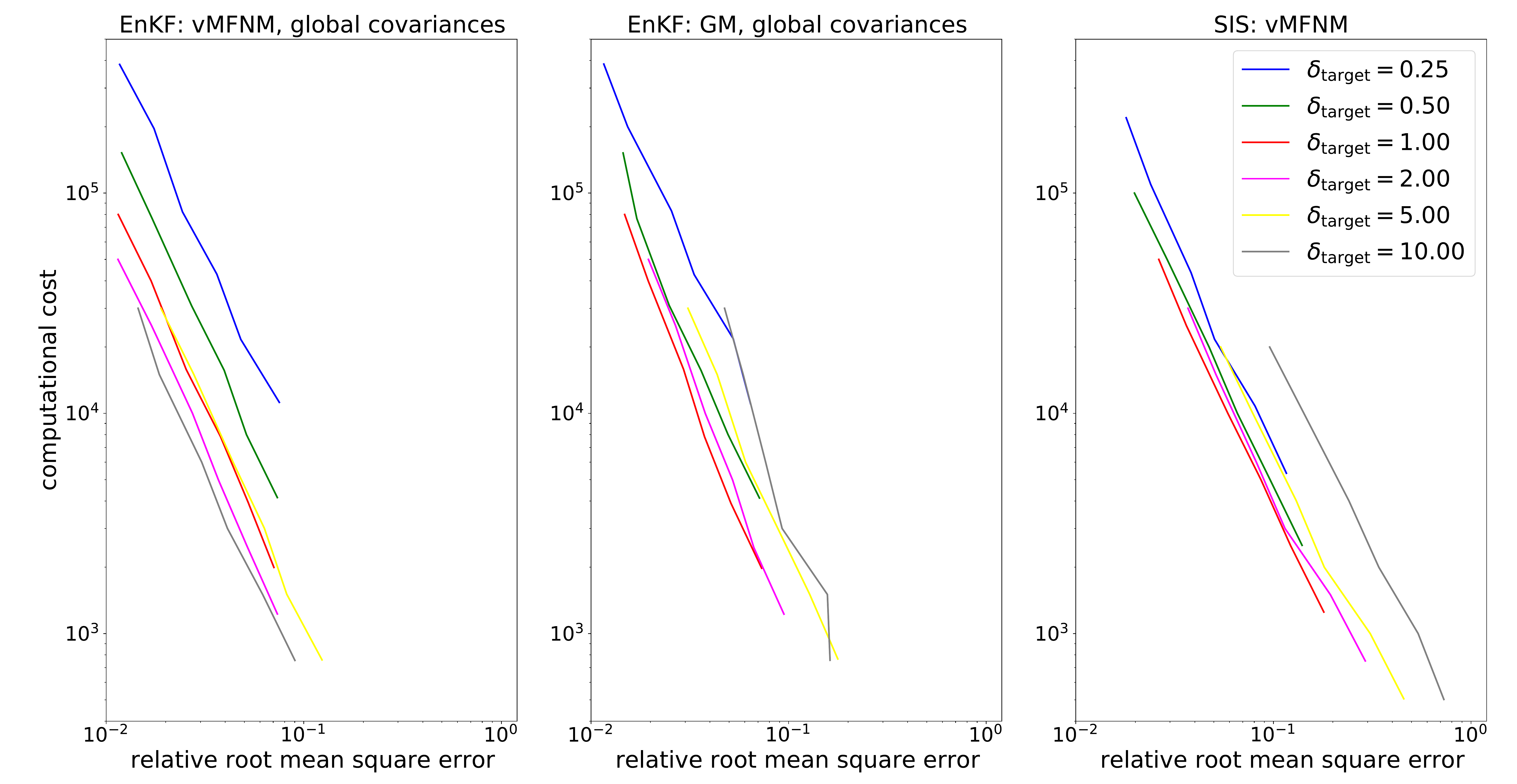}
	\caption{Convex LSF: Computational costs and relRMSE of the EnKF and SIS averaged over $500$ runs for $J\in \{250, 500, 1000, 2000, 5000, 10000\}$ samples per level and $\delta_{\mathrm{target}}\in\{0.25, 0.50, 1.00, 2.00, 5.00, 10.00\}$. Left: EnKF with vMFNM and global covariances; Middle: EnKF with GM and global covariances; Right: SIS with vMFNM. One mixture component is applied in the EnKF and SIS.}\label{Figure: example 1 error costs}
\end{figure}
\\Figure~\ref{Figure: example 1 error costs} shows the relRMSE on the horizontal axis and the computational costs on the vertical axis for the convex LSF. We observe that the EnKF reaches a slightly higher level of accuracy than SIS. For a larger number of samples per level, both methods have a smaller error. However, $\delta_{\mathrm{target}}$ has a smaller influence on the error for the EnKF than for SIS. For $\delta_{\mathrm{target}}\in\{0.25, 0.50, 1.00, 2.00\}$ the error stays constant for the EnKF with the vMFNM distribution. The error increases only for the two largest target coefficients of variations. In contrast, the error for SIS increases as $\delta_{\mathrm{target}}$ increases. For the EnKF with GM, the behaviour of the error is similar to SIS. In summary, the EnKF requires less computational costs than SIS for a fixed level of accuracy, which fits our expectations from Figure~\ref{Figure: example 1 samples}. 

\subsection{Parabolic limit-state function}
In \cite{der1998multiple}, the following parabolic LSF is proposed
\begin{align*}
G(u) = 5 - u_2 - \frac{1}{2}(u_1 - 0.1)^2.
\end{align*}
The exact probability of failure is $3.01 \cdot 10^{-3}$ \cite{Papaioannou16}. In this case, the failure domain consists of two distinct areas with high probability mass. Therefore, we apply the EnKF with local covariances. We set $\alpha=2$ and apply two mixture components to fit the GM and vMFNM distribution model in the final IS step of the EnKF. SIS is applied with the vMFNM distribution model with two mixture components. Moreover, we do not consider $J=10000$ for the ensemble size. Apart from this, we consider the same settings as for the convex LSF in the previous section.
\begin{figure}[htbp]
\centering
	\includegraphics[trim=0cm 1cm 0cm 0cm,scale=0.23]{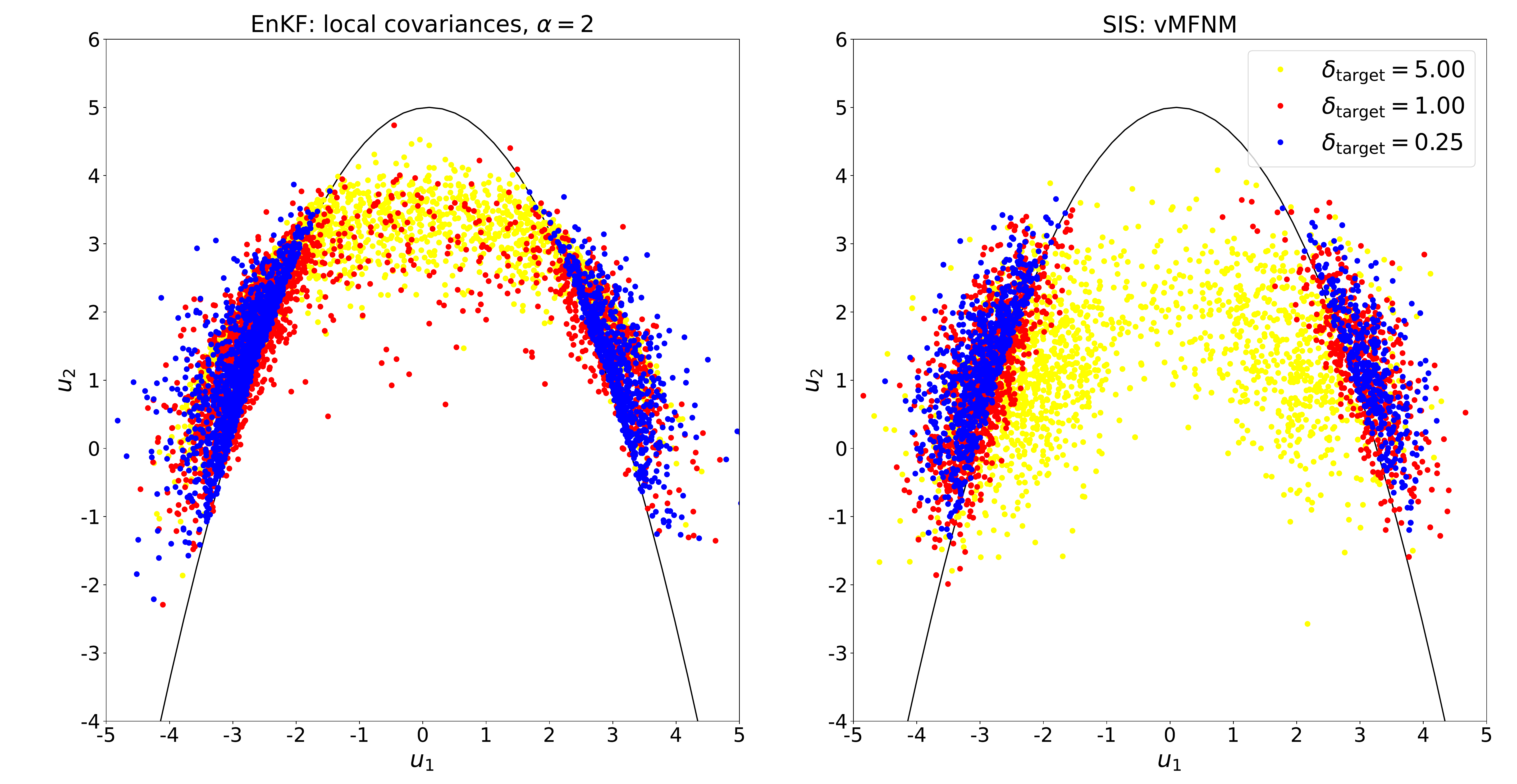}
	\caption{Parabolic LSF: Samples at the end of the iterations of the EnKF with local covariances and SIS for $\delta_{\mathrm{target}}\in\{0.25, 1.00, 5.00\}$ and for $2000$ samples per level. Two mixtures are considered for the distribution models. The black lines show the boundary of the failure domain. Left: Samples of the EnKF with local covariances and $\alpha=2$. Right: Samples of SIS with vMFNM and two mixtures.}\label{Figure: example 2 samples}
\end{figure}
\begin{figure}[htbp]
\centering
	\includegraphics[trim=0cm 1cm 0cm 0cm,scale=0.23]{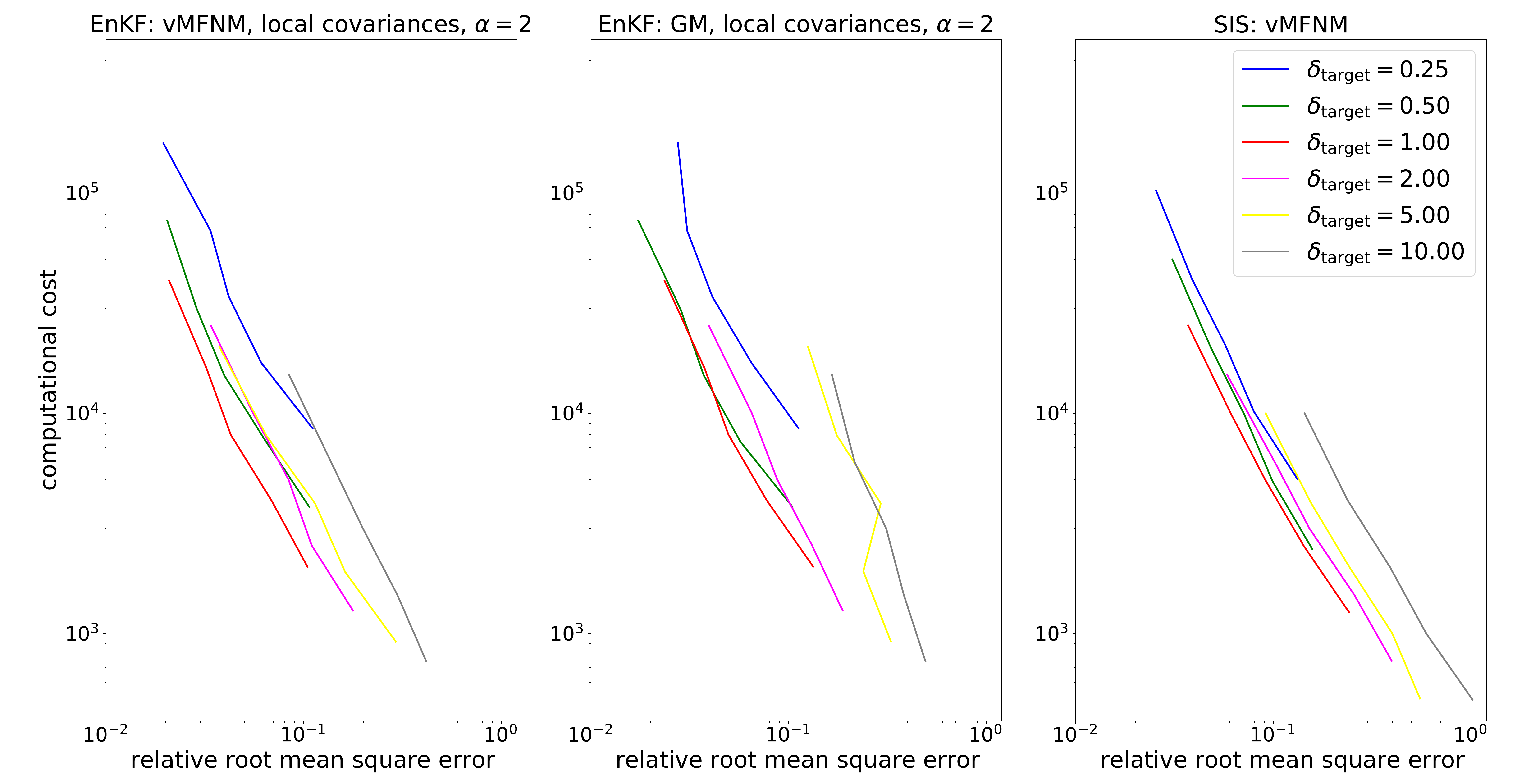}
	\caption{Parabolic LSF: Computational costs and relRMSE of the EnKF and SIS averaged over $500$ runs for $J\in\{250, 500, 1000, 2000, 5000\}$ samples per level and $\delta_{\mathrm{target}}\in\{0.25, 0.50, 1.00, 2.00, 5.00, 10.00\}$. Left: EnKF with vMFNM and local covariances, $\alpha=2$; Middle: EnKF with GM and local covariances, $\alpha=2$; Right: SIS with vMFNM. Two mixture components are applied in the EnKF and SIS.}\label{Figure: example 2 error costs}
\end{figure}
\\Figure~\ref{Figure: example 2 samples} shows the evolved samples of the final iteration of the EnKF and SIS for the parabolic LSF. We see for both methods that the generated samples concentrate near the two separated failure modes. However, for the EnKF with $\delta_{\mathrm{target}}=5$, many samples are not contained in the failure domain but are in between of the two failure modes. For SIS with $\delta_{\mathrm{target}}=5$, the samples are more concentrated around the failure modes. As in the previous example, we observe that the samples of the EnKF are more spread along the surface of the failure domain. 
\\Figure~\ref{Figure: example 2 error costs} shows the relRMSE and the computational costs for the parabolic LSF. Again, we observe that the EnKF reaches the same level of accuracy as SIS. For a larger number of samples per level, both methods have a smaller error. For $\delta_{\mathrm{target}}\in\{0.25, 0.50, 1.00\}$ the error of the EnKF is similar while for larger target coefficient of variations the error is larger for both distribution models. This behaviour is also observed in Figure~\ref{Figure: example 2 samples}. For $\delta_{\mathrm{target}}=5$, the samples do not separate clearly in the two failure modes. For SIS, we observe that the error decrease for decreasing $\delta_{\mathrm{target}}$, which implies larger computational costs. As for the convex LSF, the EnKF requires less computational costs than SIS for a fixed level of accuracy. Since $\alpha=2$ yields promising results, we do not consider a parameter study for $\alpha$, nor do we consider the adaptive approach given in Section~\ref{Subsection multi modal EnKF}. 

\subsection{Series system reliability problem}
In the third example, we consider a series system reliability problem given in \cite{waarts2000structural}, which is defined by the LSF 
\begin{align*}
G(u) = \min \left\{\begin{array}{c} 0.1(u_1 - u_2)^2-(u_1+u_2)/\sqrt{2}+3\\0.1(u_1 - u_2)^2+(u_1+u_2)/\sqrt{2}+3\\u_1 - u_2 + 7/\sqrt{2}\\u_2 - u_1 + 7/\sqrt{2}\end{array} \right\}.
\end{align*}
The corresponding probability of failure is $2.2\cdot 10^{-3}$ and the failure domain consists of four distinct modes \cite{Papaioannou16}. We apply SIS with the vMFNM distribution model with four mixtures components. The EnKF is applied with local covariances and four mixture components in the final fitting step. For this example, we consider the two approaches given in Section~\ref{Subsection multi modal EnKF} for determining the weight matrix $W$ in~\eqref{weight matrix}. We start by performing a parameter study for the parameter $\alpha$. Thereafter, we consider the adaptive approach given in Section~\ref{Subsection multi modal EnKF}, which we reference as the \emph{EnKF with adaptive local covariances}. 
\\We consider the parameter values $\alpha\in\{0.10, 0.25, 0.50, 0.75, 1.00, 2.00\}$ and apply the EnKF with local covariances and $2000$ samples per level. Figure~\ref{Figure: example 3 samples test local cov} shows the generated samples for varying $\alpha$ and $\delta_{\mathrm{target}}\in\{1.00, 5.00\}$. We observe that the samples are more concentrated for small $\alpha$. Indeed, for smaller values for $\alpha$, the samples separate into four failure modes. For $\alpha\ge 1$, nearly all samples are contained in two failure modes. In summary, for $\alpha\in\{0.25, 0.50, 0.75\}$ we expect good results since the samples capture well the four modes. For $\alpha=0.10$, the samples are too concentrated. Thus, we choose $\alpha=0.25$ and investigate the respective performance in more detail.
\begin{figure}[htbp]
\centering
	\includegraphics[trim=0cm 1cm 0cm 0cm,scale=0.23]{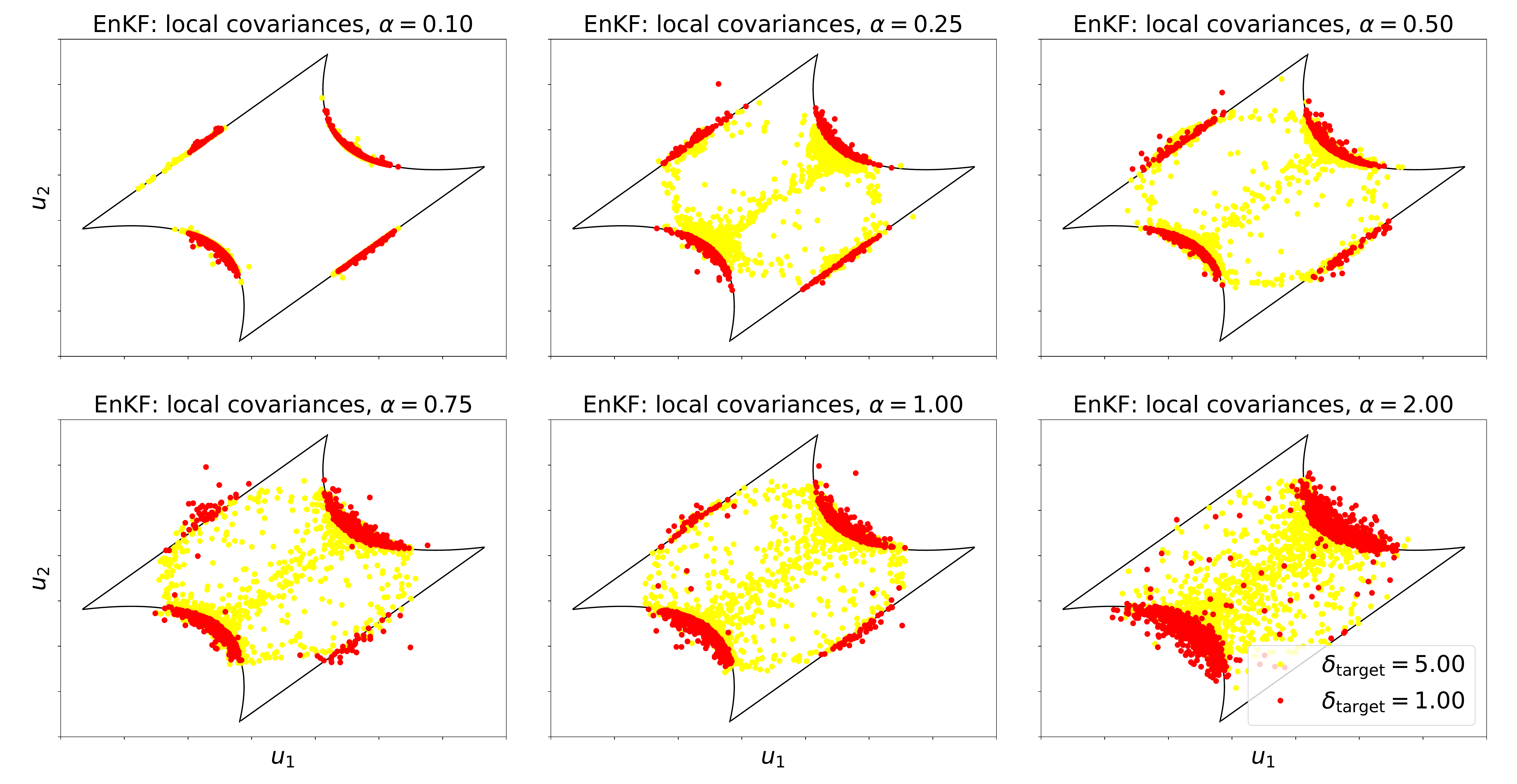}
	\caption{Series system LSF: Samples at the end of the iterations of the EnKF with local covariances for $\delta_{\mathrm{target}}\in\{1.00, 5.00\}$, $\alpha\in\{0.10, 0.25, 0.50, 0.75, 1.00, 2.00\}$ and $2000$ samples per level. The black lines show the boundary of the failure domain.}\label{Figure: example 3 samples test local cov}
\end{figure}
\begin{figure}[htbp]
\centering
	\includegraphics[trim=0cm 1cm 0cm 0cm,scale=0.23]{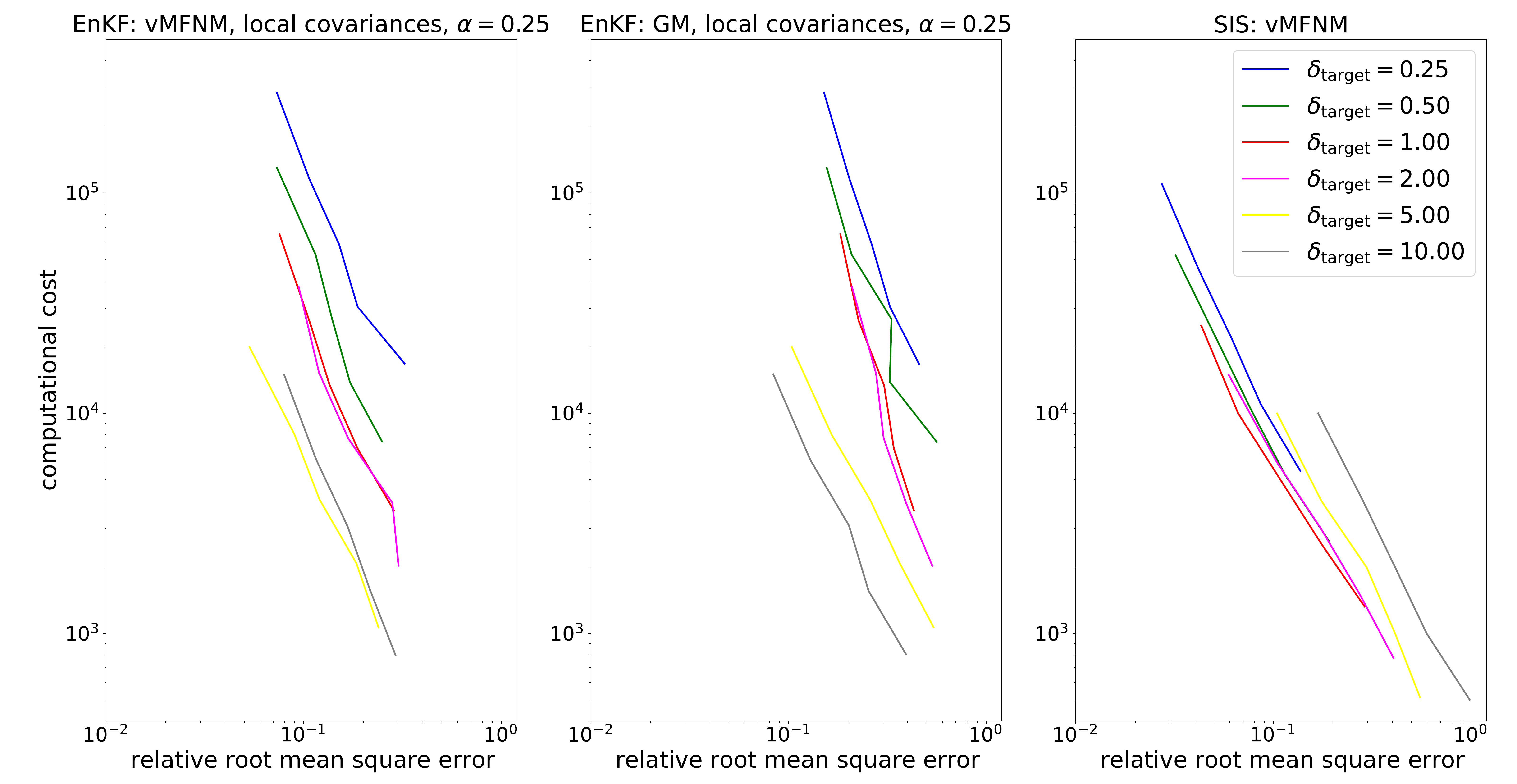}
	\caption{Series system LSF: Computational costs and relRMSE of the EnKF and SIS averaged over $500$ runs for $J\in\{250, 500, 1000, 2000, 5000\}$ samples per level and $\delta_{\mathrm{target}}\in\{0.25, 0.50, 1.00, 2.00, 5.00, 10.00\}$. Left: EnKF with vMFNM and local covariances, $\alpha=0.25$; Middle: EnKF with GM and local covariances, $\alpha=0.25$; Right: SIS with vMFNM. Four mixture components are applied in the EnKF and SIS.}\label{Figure: example 3 error costs test local cov 025}
\end{figure}
\\Figure~\ref{Figure: example 3 error costs test local cov 025} shows the performance of the EnKF with $\alpha=0.25$. The ensemble size is $J\in\{250, 500, 1000, 2000, 5000\}$ and $\delta_{\mathrm{target}}\in\{0.25, 0.50, 1.00, 2.00, 5.00, 10.00\}$ is the target coefficient of variation. We observe that a larger value for $\delta_{\mathrm{target}}$ leads to a smaller error. This is due to the fact that for small $\delta_{\mathrm{target}}$ the EnKF particles are more concentrated and do not always split into the four failure modes. Therefore, the estimates contain a bias. In particular for $\delta_{\mathrm{target}}\in\{5.00,10.00\}$, the EnKF yields a good performance. However, SIS requires less computational costs for a fixed level of accuracy.
\\In the following, we consider the adaptive approach of Section~\ref{Remark adaptive approach multi modal} to determine the weight matrix $W$. In particular, we fit the GM or vMFNM distribution model with four mixture components in each EnKF step to split the particles in a cluster with four components. Consequently, we calculate the empirical covariance matrix for each cluster and determine the weight matrix $W$ by~\eqref{weight matrix adaptive}.
\begin{figure}[htbp]
\centering
	\includegraphics[trim=0cm 1cm 0cm 0cm,scale=0.23]{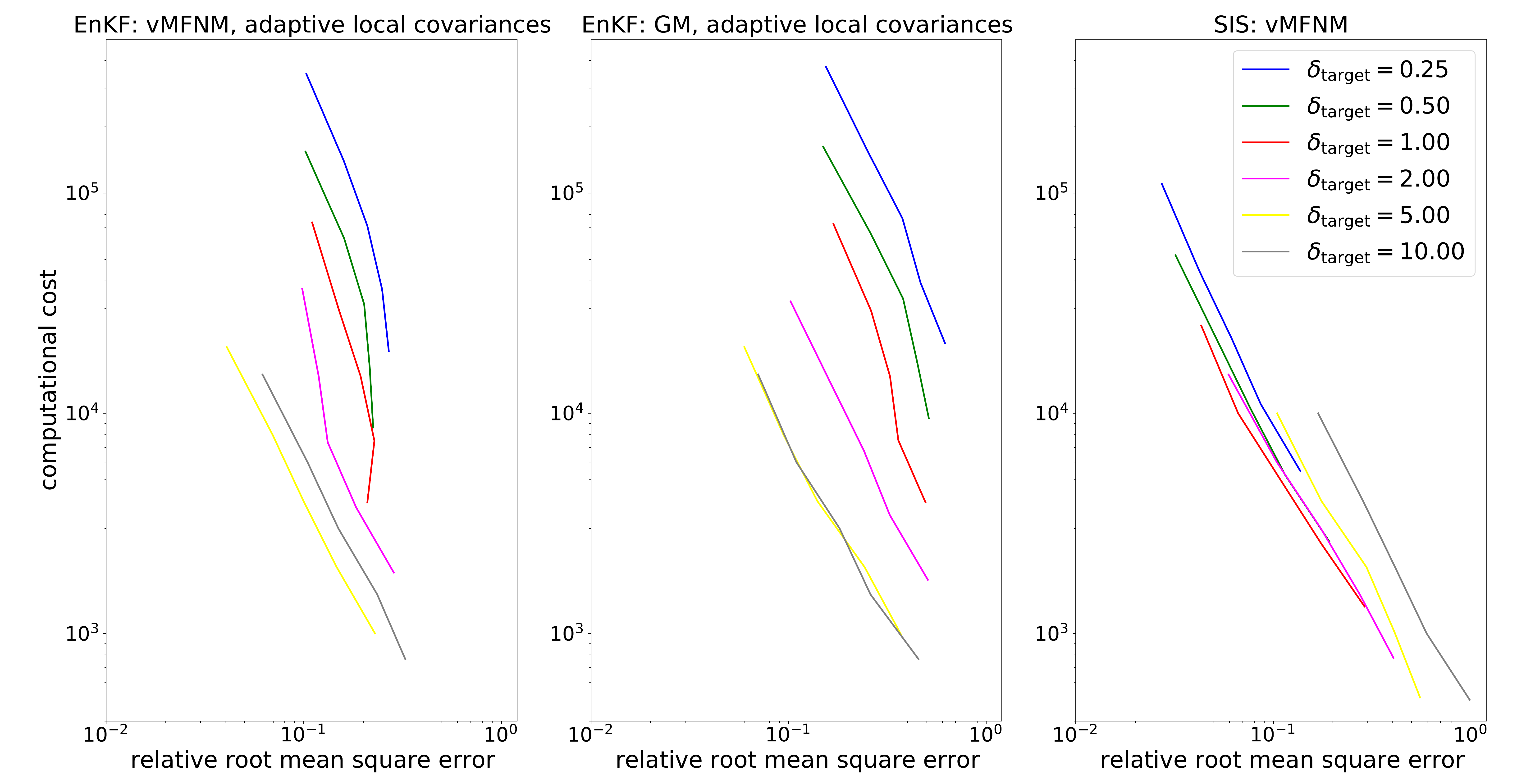}
	\caption{Series system LSF: Computational costs and relRMSE of the EnKF and SIS averaged over $500$ runs for $J\in\{250, 500, 1000, 2000, 5000\}$ samples per level and $\delta_{\mathrm{target}}\in\{0.25, 0.50, 1.00, 2.00, 5.00, 10.00\}$. Left: EnKF with vMFNM and adaptive local covariances; Middle: EnKF with GM and adaptive local covariances; Right: SIS with vMFNM. Four mixture components are applied in the EnKF and SIS.}\label{Figure: example 3 error costs adaptive local cov}
\end{figure}
\\Figure~\ref{Figure: example 3 error costs adaptive local cov} shows the relRMSE and the computational costs for the EnKF with adaptive local covariances. We observe a similar behaviour of the error as in Figure~\ref{Figure: example 3 error costs test local cov 025}. A larger value for $\delta_{\mathrm{target}}$ yields a smaller error. However, the EnKF with the vMFNM and $\delta_{\mathrm{target}}=5.00$ requires less computational costs than SIS for a fixed level of accuracy. 
\\We note that both local covariance approaches require more tempering steps. This is due to the fact that the particles interact less with their neighbours and, thus, move slower. This observation is also made in \cite{Reich19}.

\subsection{1D diffusion equation}
The final example considers the diffusion equation in the one-dimensional domain $D=(0,1)$. For $\mathbb{P}$-almost every (a.e.) $\omega\in\Omega$, we seek the weak solution $y(\cdot,\omega)\in H_0^1(D)$ such that for all $v\in H_0^1(D)$ it holds
\begin{align*}
\int_{D} a(x, \omega)\frac{\partial}{\partial x} y(x,\omega)\cdot\frac{\partial}{\partial x}v(x)\mathrm{d}x &= \int_D v(x)\mathrm{d}x.
\end{align*} 
The diffusion coefficient $a(x,\omega)=\exp(Z(x,\omega))$ is a log-normal random field. It is specified by its mean function $\mathrm{E}[a(x,\cdot)] = 1$ and standard deviation $\mathrm{Std}[a(x,\cdot)] = 0.1$. Thus, the mean function of $Z$ is $\mu_Z = \log(\mathrm{E}[a(x,\cdot)]) - {\sigma_Z^2}/{2}$ and the variance is given by $\sigma_Z^2 = \log\left(({\mathrm{Std}[a(x,\cdot)]^2 + \mathrm{E}[a(x,\cdot)]^2})/{\mathrm{E}[a(x,\cdot)]^2}\right)$. Moreover, we assume that $Z$ has an exponential type covariance function $c(x,y) = \sigma_Z^2\exp\left(-{\Vert x -y\Vert_1}/{\lambda}\right)$ with correlation length $\lambda = 0.01$. The truncated Karhunen--Lo\`eve (KL) expansion
\begin{align*}
Z_d(x,\omega) = \mu_Z + \sigma_Z \sum_{m=1}^d \sqrt{\nu_m}\theta_m(x) U_m(\omega)
\end{align*}
gives an approximation to the infinite-dimensional random field $Z$. Thus, we define $a_d = \exp(Z_d)$ as an approximation for $a$. The associated solution of the weak form is denoted by $y_d$. The eigenpairs $(\nu_m, \theta_m)$ can be analytically calculated as explained in \cite[Section 2.3.3]{Ghanem91}. Moreover, $U:=\{U_m\}_{m=1}^d$ are independent standard normal Gaussian random variables. We set $d=150$ which captures $87\%$ of the variability of $\log(a)$. 
\\In addition, we approximate the solution $y_d$ by piecewise linear, continuous finite elements on a uniform grid with mesh size $h=1/512$. The finite element approximation is denoted by $y_h$. Finally, we call the event $\omega\in\Omega$ a failure event if the solution $y_h(\cdot,\omega)$ is larger than $0.535$ at $x=1$. This gives the LSF
\begin{align*}
G(U(\omega)):= 0.535 - y_h(x=1,\omega).
\end{align*}
By crude Monte Carlo sampling with $2\cdot 10^8$ samples, the probability of failure is estimated as $P_{f} = 1.682\cdot 10^{-4}$. In the following, this value is referred to as the reference solution. We note that the truncation of the KL expansion and the discretization parameter $h$ induces an error in $G$. Thus, the probability of failure $P_{f}$ is an approximation to the exact one which requires the exact solution $y$. Since we always consider a fixed discretization level and fix the number of KL terms, the error is not present in the estimates. For an error analysis with respect to the discretization size $h$, we refer to \cite{Wagner20_2}. 
\\The probability of failure is estimated by the EnKF and SIS. The estimation is performed for $J\in\{250, 500, 1000, 2000\}$ samples per level and target coefficient of variation equal to $\delta_{\mathrm{target}}\in\{0.25, 0.50, 1.00, 2.00, 5.00, 10.00\}$. For the EnKF, we use the vMFNM as distribution model with one mixture. Moreover, we apply global covariances since we expect one single failure mode. The GM distribution is not considered, since it does not perform well in high dimensions. Similar we apply SIS with sampling from the vMFNM distribution with one mixture.
\begin{figure}[htbp]
\centering
	\includegraphics[trim=0cm 1cm 0cm 0cm,scale=0.23]{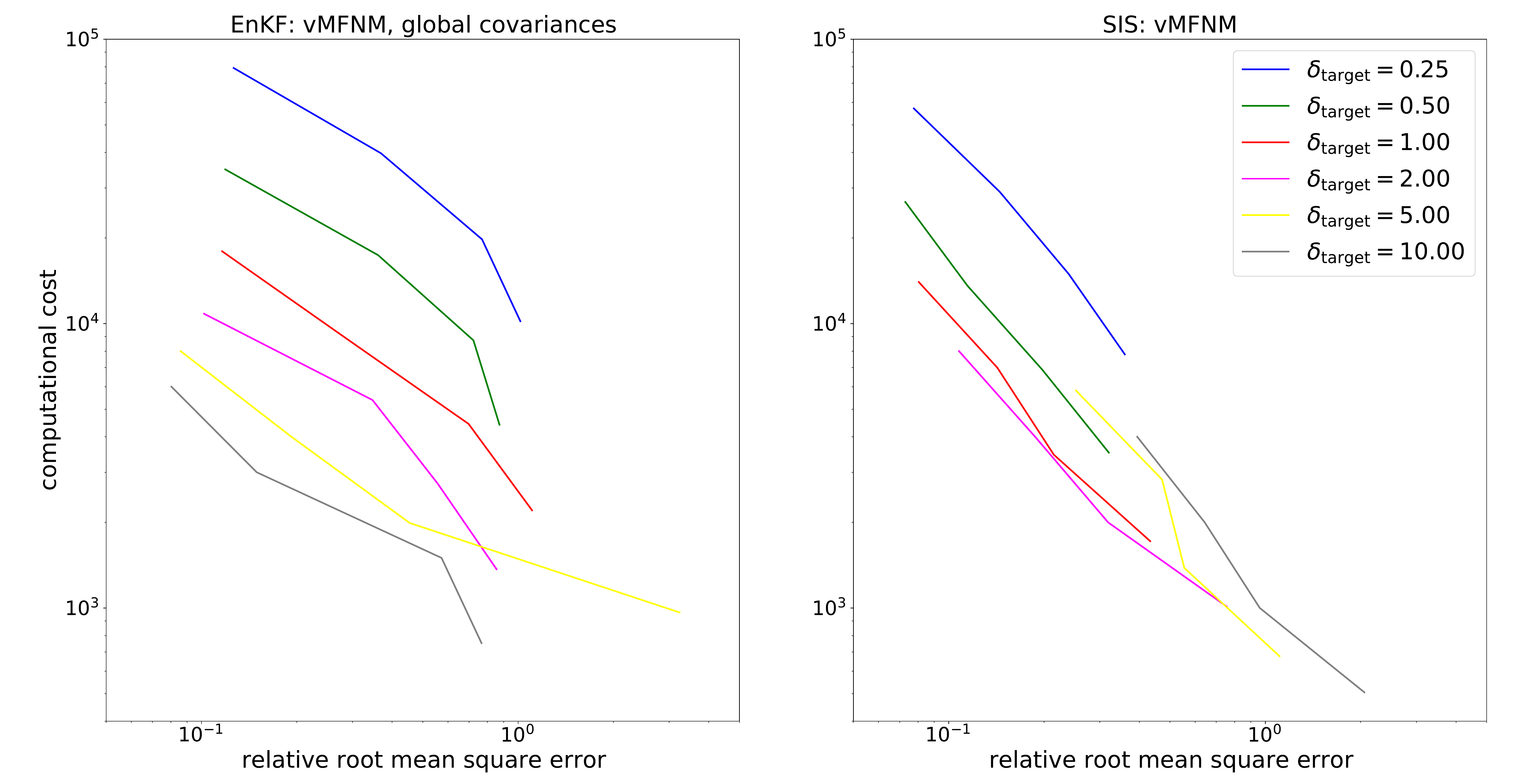}
	\caption{Diffusion equation: Computational costs and relRMSE of the EnKF and SIS averaged over $100$ runs for $J\in\{250, 500, 1000, 2000\}$ samples per level and $\delta_{\mathrm{target}}\in\{0.25, 0.50, 1.00, 2.00, 5.00, 10.00\}$. Left: EnKF with vMFNM and global covariances; Right: SIS with vMFNM. One mixture component is applied in the EnKF and SIS.}\label{Figure: example 4 error costs}
\end{figure}
\\Figure~\ref{Figure: example 4 error costs} shows the relRMSE and the computational costs for the diffusion equation problem. The EnKF yields the same level of accuracy as SIS. Indeed, the EnKF yields the smallest error with the largest target coefficient of variation. We note that $\delta_{\mathrm{target}}=10$ requires only one step to reach the stopping criterion. The surface of the failure domain might be highly nonlinear. If $\delta_{\mathrm{target}}$ is large, the final EnKF particles are more spread, which yields a better fitting distribution in this particular case and a smaller error. In contrast, a smaller value for $\delta_{\mathrm{target}}$ leads to a smaller error for SIS. This is the complete opposite observation as for the EnKF. This is due to the fact that for small $\delta_{\mathrm{target}}$, two consecutive densities in SIS are more similar and the estimation of the probability of failure is more robust. In summary, we conclude that the EnKF requires less computational costs for a fixed level of accuracy.

\section{Conclusion and Outlook}\label{Section Conclusion}
We introduce a novel sampling method for estimating small probabilities of failure that employs the EnKF for inverse problems. The proposed method reformulates the rare event problem as an inverse problem by concatenating the ReLU function with the LSF. For this reformulation, the EnKF is applied to generate failure samples in an adaptive manner. We have shown that the EnKF densities define an alternative sequence of piecewise smooth approximations of the optimal IS density as compared to the sequence employed in SIS. Consequently, a distribution model is fitted with the generated samples and the probability of failure is estimated with IS. We have employed two density models in the IS step, the GM model and the vMFNM distribution model. The GM model is applicable for low-dimensional parameter spaces, while its performance deteriorates for moderate- and high-dimensional spaces. In contrast, the vMFNM distribution model is applicable for moderately high-dimensional parameter spaces. To ensure that the parameters of the vMFNM distribution model are estimated accurately, the number of samples per level has to be large for high-dimensional parameter spaces.
\\For affine linear LSFs, we have derived the particle dynamic for the continuous time limit of the EnKF update. Under the assumption that the EnKF is applied without noise, we have proven that the ensemble mean converges to a convex combination of the most likely failure point and the mean of the optimal IS density in the large particle and large time limit.
\\To handle multi-modal failure domains, we localise the covariance matrices in the EnKF update around each particle. This localisation can be made adaptively using a clustering approach.
\\In numerical experiments, we compare the EnKF with SIS in terms of the relative root mean square error and the required computational costs. For single modal failure domains, the EnKF requires less computational costs than SIS for a fixed level of accuracy. However, for multi-modal failure domains, the application of the EnKF is not straightforward and SIS yields a better performance in some cases. 
\\For future work, our manuscript can be used as a starting point to analyse the EnKF for rare events for more general settings. In particular, the analysis for nonlinear LSFs and the analysis of the mean field limit of the EnKF with noise is still an open question. Moreover, the combination of the EnKF with a multilevel strategy would be beneficial for LSFs which can be approximated by a hierarchy of discretization levels. Finally, the EnKF algorithm could be extended for estimating the probabilities of failure associated with multivariate outputs. This approach could be beneficial in cases where the multiple outputs are based on a single model solve.

\appendix
\section{Proofs of Section~\ref{Sec EnKF for affine linear LSFs}}\label{Appendix}
\begin{proof}[Proof of Theorem~\ref{Theorem linear noise free}]
The authors of \cite{Schillings16} show for the continuous time limit $h\rightarrow 0$ of ~\eqref{EnKF rare events uk} and the noise free case that the particles satisfy the flow
\begin{align}
\frac{\mathrm{d}u^{(j)}}{\mathrm{d}t} = \frac{1}{J}\sum_{k=1}^J\langle \widetilde{G}(u^{(k)})-\overline{\widetilde{G}}, -\widetilde{G}(u^{(j)})\rangle (u^{(k)}-\overline{u}),\label{cts time limit}
\end{align}
where we used that $y^{\dagger}=0$ and $\Gamma = 1$. If $\widetilde{G}(u^{(j)})=0$, ~\eqref{cts time limit} implies that
\begin{align*}
\frac{\mathrm{d}u^{(j)}}{\mathrm{d}t}=0.
\end{align*}
Thus, failure particles do not move. Now we consider the case that $\widetilde{G}(u^{(j)})= G(u^{(j)})> 0$. The ensemble mean $\overline{\widetilde{G}}$ can be expressed as
\begin{align*}
\overline{\widetilde{G}}=\frac{1}{J}\sum_{k=1}^J\widetilde{G}(u^{(k)}) = \frac{1}{J}\sum_{k\in S} G(u^{(k)}) &= \frac{1}{J}\sum_{k=1}^{J} G(u^{(k)}) - \frac{1}{J}\sum_{k\in F} G(u^{(k)}) = G(\overline{u}) - C,
\end{align*} 
since $G$ is affine linear for all $u$ and we define $C:=\frac{1}{J}\sum_{k\in F} G(u^{(k)})<0$. Splitting up~\eqref{cts time limit} into the safe and failure states and using that $\widetilde{G}(u^{(k)}) = G(u^{(k)})$ for $k\in S$ and $\widetilde{G}(u^{(k)}) =0$ for $k\in F$, we get
\begin{align}
\frac{\mathrm{d}u^{(j)}}{\mathrm{d}t} = \frac{1}{J}& \left(\sum_{k\in S}\langle G(u^{(k)})-G(\overline{u})+C, -G(u^{(j)})\rangle(u^{(k)}-\overline{u})\right. \label{sum 1}
\\ &+ \left.\sum_{k\in F}\langle -G(\overline{u})+C, -G(u^{(j)})\rangle(u^{(k)}-\overline{u})\right).\label{sum 2}
\end{align}
The sum in~\eqref{sum 1} is equal to
\begin{align*}
&\sum_{k\in S}\langle G(u^{(k)})-G(\overline{u}), -G(u^{(j)})\rangle(u^{(k)}-\overline{u}) + \sum_{k\in S}\langle C, -G(u^{(j)})\rangle(u^{(k)}-\overline{u}):= S_1 + S_2,
\end{align*}
while the sum in~\eqref{sum 2} is equal to
\begin{align*}
&\sum_{k\in F}\langle G(u^{(k)})-G(\overline{u}), -G(u^{(j)})\rangle(u^{(k)}-\overline{u}) + \sum_{k\in F}\langle C-G(u^{(k)}), -G(u^{(j)})\rangle(u^{(k)}-\overline{u})
\\ &:= F_1 + F_2,
\end{align*}
where we have used the linearity of the scalar product. Adding $S_1$ and $F_1$ and multiplying with $1/J$ gives
\begin{align}
\frac{1}{J}\left(S_1 + F_1\right) &= \frac{1}{J} \sum_{k=1}^{J}\langle G(u^{(k)})-G(\overline{u}), -G(u^{(j)})\rangle(u^{(k)}-\overline{u})\notag
\\&=-C_{\mathrm{uu}}(\mathbf{u})D_u\left(\frac{1}{2}G(u^{(j)})^2\right),\label{sum 3}
\end{align} 
where we have applied~\eqref{EnKf inverse sample flow} since $G$ is affine linear. Adding the remaining parts gives
\begin{align}
\frac{1}{J}\left(S_2 + F_2\right) &= \frac{\langle C, -G(u^{(j)})\rangle}{J}\sum_{k=1}^J(u^{(k)}-\overline{u}) + \frac{1}{J}\sum_{k\in F}\langle G(u^{(k)}), G(u^{(j)})\rangle(u^{(k)}-\overline{u}) \notag
\\&= \frac{G(u^{(j)})}{J}\sum_{k\in F}G(u^{(k)})(u^{(k)}-\overline{u}),\label{sum 4}
\end{align}
since $\sum_{k=1}^J\left(u^{(k)}-\overline{u}\right) =0$. Adding~\eqref{sum 3} and~\eqref{sum 4} gives the desired result.
\end{proof}

\subsection{Proof of Section~\ref{Sec no failure}}

\begin{proof}[Proof of Lemma~\ref{Lemma FORM}]
We follow the proof of \cite[Lemma 3.2]{Garbuno20} and adjust it to the rare event setting. The second part of~\eqref{gradient flow rare events} is zero for all $t\ge0$ since all initial particles are in the safe domain and, at the time point a particle reaches the failure surface, it does not move anymore. Therefore, the dynamic of all particles satisfies for all $t\ge 0$ 
\begin{align*}
\frac{\mathrm{d}u^{(j)}}{\mathrm{d}t} = -C_{\mathrm{uu}}(\mathbf{u})D_u\left(\frac{1}{2}G(u^{(j)})^2\right) = -C_{\mathrm{uu}}(\mathbf{u})\left(aa^Tu^{(j)}-ab\right).
\end{align*}
The large particle limit $J\rightarrow\infty$ leads to the mean field equation at $t\ge 0$
\begin{align}
\frac{\mathrm{d}u(t)}{\mathrm{d}t} = -C(t)\left(aa^Tu(t)-ab\right),\label{mean field equation}
\end{align}
where $u(t)$ is a realisation of $U(t)$. With~\eqref{mean field equation}, we can derive the dynamic of the mean and covariance matrix. The ensemble mean satisfies
\begin{align}
\frac{\mathrm{d}m(t)}{\mathrm{d}t} = -C(t)\left(aa^Tm(t)-ab\right).\label{mean dynamic}
\end{align}
By defining $e(t)=U(t)-m(t)$, we get
\begin{align}
\frac{\mathrm{d}e(t)}{\mathrm{d}t} = -C(t)aa^Te(t).\label{e dot}
\end{align}
For the covariance it holds that $C(t) = \mathbb{E}[e(t)\otimes e(t)]$. Differentiating $C(t)$ with respect to $t$ and plugging in~\eqref{e dot}, it follows
\begin{align}
\frac{\mathrm{d}C(t)}{\mathrm{d}t}=\mathbb{E}\left[\frac{\mathrm{d}e(t)}{\mathrm{d}t}\otimes e + e\otimes\frac{\mathrm{d}e(t)}{\mathrm{d}t}\right] = -2C(t)aa^TC(t).\label{Covariance dynamic 2}
\end{align}
From~\eqref{Covariance dynamic 2}, it follows for the inverse of the covariance matrix
\begin{align}
\frac{\mathrm{d}C^{-1}(t)}{\mathrm{d}t} = -C^{-1}(t)\left(\frac{\mathrm{d}C(t)}{\mathrm{d}t}\right)C^{-1}(t) = 2aa^T.\label{Covariance inverse dynamic}
\end{align}
With the initial condition $C(0)=\mathrm{Id}_d$ and from~\eqref{Covariance inverse dynamic}, it follows
\begin{align*}
C(t) = \left(\mathrm{Id}_d + 2aa^Tt\right)^{-1} = \begin{pmatrix}
1/(1+2t) & \\
 & \mathrm{Id}_{d-1}
\end{pmatrix},
\end{align*}
where $\mathrm{Id}_{d-1}\in\mathbb{R}^{(d-1)\times (d-1)}$ is the identity matrix. 
Inserting the expression of the covariance matrix $C(t)$ in~\eqref{mean dynamic} gives
\begin{align*}
\frac{\mathrm{d}m_1(t)}{\mathrm{d}t} = \frac{b-m_1(t)}{1+2t},\quad\quad \frac{\mathrm{d}m_i(t)}{\mathrm{d}t} = 0,\quad\text{ for } i=2,\dots,d.
\end{align*}
With the initial condition $m(0)=0$, the entries of the mean are given by 
\begin{align*}
m_1(t) = b\left(1 - \frac{1}{\sqrt{2t+1}}\right), \quad\quad m_i(t)=0, \quad\text{ for } i=2,\dots,d.
\end{align*}
Since $\lim_{t\rightarrow\infty} m_1(t)=b$, we conclude that $\lim_{t\rightarrow\infty} m(t) = u^{\mathrm{MLFP}}$, which is the desired result.
\end{proof}

\subsection{Proofs of Section~\ref{Sec with failure samples}}\label{Appendix sec with failure}

\begin{proof}[Proof of Lemma~\ref{Lemma mean field limit}]
Theorem~\ref{Theorem linear noise free} gives the continuous time limit of the particle dynamic. We consider the mean field limit $J\rightarrow\infty$ for the two parts in~\eqref{gradient flow rare events} separately. For the first part, it holds that
\begin{align*}
\lim_{J\rightarrow\infty} -C_{\mathrm{uu}}(\mathbf{u})D_u\left(\frac{1}{2}G(u^{(j)})^2\right) = -C(t)D_u\left(\frac{1}{2}G(u^{(j)})^2\right).
\end{align*}
Now, we consider the second part in~\eqref{gradient flow rare events}. We split the sum into two parts as
\begin{align}
\frac{G(u^{(j)})}{J}&\sum_{k\in F}G(u^{(k)})(u^{(k)}-\overline{u})\notag 
\\ &= G(u^{(j)}) \frac{\vert F\vert}{J}\left(\frac{1}{\vert F\vert} \sum_{k\in F} G(u^{(k)})u^{(k)} - \frac{\overline{u}}{\vert F\vert} \sum_{k\in F} G(u^{(k)})\right).\label{eq 1}
\end{align}
At first, we see that $\lim_{J\rightarrow\infty} \vert F \vert/J = P_f$ since $\vert F\vert$ is the number of failure particles in the initial ensemble. The limit of the first sum in~\eqref{eq 1} is
\begin{align}
 \lim_{J\rightarrow\infty} \frac{1}{\vert F\vert}\sum_{k\in F} G(u^{(k)})u^{(k)} &= \mathbb{E}[G(U)U \mid G(U)<0]\notag
\\ &= \mathbb{E}[U_1U\mid U_1<b] - b \mathbb{E}[U\mid U_1<b].\label{eq expectation}
\end{align}
The second tern in~\eqref{eq expectation} is equal to $-bu^{\mathrm{opt}}$. Since the components of $U$ are independent, it holds that $\mathbb{E}[U_1U\mid U_1<b] = \left(\mathbb{E}[U_1^2\mid U_1< b], 0,\dots,0\right)^T$.
\\Since $\mathrm{Var}[U_1\mid U_1< b] = \mathbb{E}[U_1^2\mid U_1< b] - \mathbb{E}[U_1\mid U< b]^2$ we conclude that
\begin{align*}
\mathbb{E}[U_1^2\mid U_1<b] & = \mathrm{Var}[U_1\mid U_1<b] + \mathbb{E}[U_1\mid U_1\le b]^2
\\ &= 1 - b\frac{\varphi(b)}{\Phi(b)} - \frac{\varphi^2(b)}{\Phi^2(b)} + \left(-\frac{\varphi(b)}{\Phi(b)}\right)^2 = 1 + bu_1^{\mathrm{opt}},
\end{align*}
where we used the formula of the mean and variance of a truncated Gaussian \cite[Section 10.1]{Johnson94}. In summary we get
\begin{align*}
\lim_{J\rightarrow\infty} \frac{1}{\vert F\vert}\sum_{k\in F} G(u^{(k)})u^{(k)} = \left(1,0,\dots,0\right)^T.
\end{align*}
For the second sum in~\eqref{eq 1}, we get
\begin{align*}
\lim_{J\rightarrow\infty} - \frac{\overline{u}}{\vert F\vert}\sum_{k\in F} G(u^{(k)}) = - m(t)\mathbb{E}[G(U)\mid G(U)<0] = -m(t)(u_1^{\mathrm{opt}} - b)
\end{align*}
by linearity of $G$. In summary we conclude that
\begin{align*}
\lim_{J\rightarrow\infty} \frac{G(u^{(j)})}{J}\sum_{k\in F}G(u^{(k)})(u^{(k)}-\overline{u}) = G(u^{(j)}) P_f\left((1,0,\dots,0)^T -m(t)(u_1^{\mathrm{opt}} - b)\right).
\end{align*}
Together with the first limit, we get for $J\rightarrow\infty$ that
\begin{align*}
\frac{\mathrm{d}u^{(j)}}{\mathrm{d}t} = -C(t)D_u\left(\frac{1}{2}G(u^{(j)})^2\right) + G(u^{(j)}) P_f\left((1,0,\dots,0)^T -m(t)(u_1^{\mathrm{opt}} - b)\right),
\end{align*}
which is the desired result.
\end{proof}

\begin{proof}[Proof of Lemma~\ref{lemma covariance}]
We consider $d=1$. For the variance it holds that
\begin{align*}
C(t) = \mathrm{Var}[U(t)] = &\mathbb{E}[U(t)^2] - \mathbb{E}[U(t)]^2
\\= & (1-P_f)\mathbb{E}[U(t)^2\mid U(t)\ge b] + P_f\mathbb{E}[U(t)^2\mid U(t)< b]
\\ &- \left((1-P_f)m_S(t) + P_f m_F(t)\right)^2.
\end{align*}
Again, we use
\begin{align*}
\mathbb{E}[U(t)^2\mid U(t)\ge b] = \mathrm{Var}[U(t)\mid U(t)\ge b] + \mathbb{E}[U(t)\mid U(t)\ge b]^2 > m_S(t)^2,
\end{align*}
since the variance is always positive as long as $U_S(t)$ is not collapsed to a single point. Using that $\mathbb{E}[U(t)^2\mid U(t)< b] = 1+bu^{\mathrm{opt}}$ and $m_F = u^{\mathrm{opt}}$, we get
\begin{align*}
C(t) &> (1-P_f)m_S(t)^2 + P_f(1+bu^{\mathrm{opt}})-((1-P_f)m_S(t)+P_f u^{\mathrm{opt}})^2.
\end{align*}
\end{proof}

\begin{proof}[Proof of Theorem~\ref{Thm mean failure}]
It remains to show that
\begin{align}
C(t) - P_f\left(1 -m(t)(u^{\mathrm{opt}} - b)\right) >0.\label{ineq cova}
\end{align}
We check that this is true if $m_S(t)>u^{\mathrm{MLFP}}$. We start by using Lemma~\ref{lemma covariance} and splitting up $m(t)$. Thus,
\begin{align*}
C(t) & - P_f\left(1 -m(t)(u^{\mathrm{opt}} - b)\right)
\\ > & (1-P_f)m_S(t)^2 + P_f(1+bu^{\mathrm{opt}})-((1-P_f)m_S(t)+P_f u^{\mathrm{opt}})^2
\\ &- P_f\left(1 - \left((1-P_f)m_S(t)+ P_f u^{\mathrm{opt}}\right)(u^{\mathrm{opt}} - b)\right)
\\ = & P_f(1 - P_f)\left(m_S(t)^2 + m_S(t)(-u^{\mathrm{opt}}-b) + bu^{\mathrm{opt}}\right) := P_f(1 - P_f) f(m_S(t)).
\end{align*}
The quadratic function $f$ has the roots $m_{S,1} = b$ and $m_{S,2}=u^{\mathrm{opt}}$. Moreover, it holds that $f'(b)>0$ since $b>u^{\mathrm{opt}}$. It follows that $f(m_S(t))>0$ for $m_S(t)>b$. Together with $P_f(1-P_f)>0$,~\eqref{ineq cova} holds for all $t\ge 0$ and, thus,
\begin{align*}
\lim_{t\rightarrow\infty} m_S(t) = u^{\mathrm{MLFP}}.
\end{align*}
\end{proof}

\bibliographystyle{plain}     
\bibliography{literature}   

\begin{thebibliography}{10}

\bibitem{Agapiou17}
{\sc S.~Agapiou, O.~Papaspiliopoulos, D.~Sanz-Alonso, and A.~M. Stuart}, {\em
  Importance sampling: intrinsic dimension and computational cost}, Statistical
  Science, 32 (2017), pp.~405--431, \url{https://doi.org/10.1214/17-STS611}.

\bibitem{Agarwal17}
{\sc A.~Agarwal, S.~de~Marco, E.~Gobet, and G.~Liu}, {\em Rare event simulation
  related to financial risks: efficient estimation and sensitivity analysis},
  HAL,  (2017),
  \url{https://hal-polytechnique.archives-ouvertes.fr/hal-01219616}.
\newblock Working paper.

\bibitem{Au01}
{\sc S.-K. Au and J.~L. Beck}, {\em Estimation of small failure probabilities
  in high dimensions by subset simulation}, Probabilistic Engineering
  Mechanics, 16 (2001), pp.~263--277,
  \url{https://doi.org/10.1016/S0266-8920(01)00019-4}.

\bibitem{Au14}
{\sc S.-K. Au and Y.~Wang}, {\em Engineering Risk Assessment with Subset
  Simulation}, John Wiley \& Sons, Ltd, 2014,
  \url{https://doi.org/10.1002/9781118398050}.

\bibitem{Beskos13}
{\sc A.~Beskos, A.~Jasra, N.~Kantas, and A.~Thiery}, {\em On the convergence of
  adaptive sequential {Monte Carlo} methods}, The Annals of Applied
  Probability, 26 (2016), pp.~1111--1146,
  \url{https://doi.org/10.1214/15-AAP1113}.

\bibitem{Bloemker19}
{\sc D.~Bl{\"o}mker, C.~Schillings, P.~Wacker, and S.~Weissmann}, {\em Well
  posedness and convergence analysis of the ensemble {Kalman} inversion},
  Inverse Problems, 35 (2019), p.~085007,
  \url{https://doi.org/10.1088/1361-6420/ab149c}.

\bibitem{Dashti2017}
{\sc M.~Dashti and A.~M. Stuart}, {\em The {Bayesian} approach to inverse
  problems}, in Handbook of Uncertainty Quantification, R.~Ghanem, D.~Higdon,
  and H.~Owhadi, eds., Springer, Cham, 2017, pp.~311--428,
  \url{https://doi.org/10.1007/978-3-319-12385-1_7}.

\bibitem{Angelis15}
{\sc M.~de~Angelis, E.~Patelli, and M.~Beer}, {\em Advanced line sampling for
  efficient robust reliability analysis}, Structural Safety, 52 (2015),
  pp.~170--182, \url{https://doi.org/10.1016/j.strusafe.2014.10.002}.

\bibitem{Moral06}
{\sc P.~Del~Moral, A.~Doucet, and A.~Jasra}, {\em Sequential {Monte Carlo}
  samplers}, Journal of the Royal Statistical Society. Series B (Statistical
  Methodology), 68 (2006), pp.~411--436,
  \url{https://doi.org/10.1111/j.1467-9868.2006.00553.x}.

\bibitem{der1998multiple}
{\sc A.~Der~Kiureghian and T.~Dakessian}, {\em Multiple design points in first
  and second-order reliability}, Structural Safety, 20 (1998), pp.~37--49,
  \url{https://doi.org/10.1016/S0167-4730(97)00026-X}.

\bibitem{Kiureghian86}
{\sc A.~Der~Kiureghian and P.-L. Liu}, {\em Structural reliability under
  incomplete probability information}, Journal of Engineering Mechanics, 112
  (1986), pp.~85--104,
  \url{https://doi.org/10.1061/(ASCE)0733-9399(1986)112:1(85)}.

\bibitem{Doucet11}
{\sc A.~Doucet and A.~M. Johansen}, {\em A tutorial on particle filtering and
  smoothing: fifteen years later}, in The Oxford Handbook of Nonlinear
  Filtering, D.~Crisan and B.~Rozovskii, eds., Oxford University Press, Oxford,
  2011, pp.~656--704.

\bibitem{Dovera11}
{\sc L.~Dovera and E.~Della~Rossa}, {\em Multimodal ensemble {Kalman} filtering
  using {Gaussian} mixture models}, Computational Geosciences, 15 (2011),
  pp.~307--323, \url{https://doi.org/10.1007/s10596-010-9205-3}.

\bibitem{Ernst15}
{\sc O.~G. Ernst, B.~Sprungk, and H.-J. Starkloff}, {\em Analysis of the
  ensemble and polynomial chaos {Kalman} filters in {Bayesian} inverse
  problems}, SIAM/ASA Journal on Uncertainty Quantification, 3 (2015),
  pp.~823--851, \url{https://doi.org/10.1137/140981319}.

\bibitem{Evensen06}
{\sc G.~Evensen}, {\em Data Assimilation: The Ensemble {Kalman} Filter},
  Springer, Berlin, Heidelberg, 2~ed., 2006,
  \url{https://doi.org/10.1007/978-3-642-03711-5}.

\bibitem{Fishman96}
{\sc G.~S. Fishman}, {\em {Monte Carlo}: Concepts, Algorithms and
  Applications}, Springer Series in Operations Research, Springer, New York,
  NY, 1~ed., 1996, \url{https://doi.org/10.1007/978-1-4757-2553-7}.

\bibitem{Schnatter06}
{\sc S.~Fr{\"u}hwirth-Schnatter}, {\em Finite Mixture and {M}arkov Switching
  Models}, Springer Series in Statistics, Springer, New York, NY, 1~ed., 2006,
  \url{https://doi.org/10.1007/978-0-387-35768-3}.

\bibitem{Garbuno20}
{\sc A.~Garbuno-Inigo, F.~Hoffmann, W.~Li, and A.~M. Stuart}, {\em Interacting
  {Langevin} diffusions: gradient structure and ensemble {Kalman} sampler},
  SIAM Journal on Applied Dynamical Systems, 19 (2020), pp.~412--441,
  \url{https://doi.org/10.1137/19M1251655}.

\bibitem{Geyer19}
{\sc S.~Geyer, I.~Papaioannou, and D.~Straub}, {\em Cross entropy-based
  importance sampling using {Gaussian} densities revisited}, Structural Safety,
  76 (2019), pp.~15--27, \url{https://doi.org/10.1016/j.strusafe.2018.07.001}.

\bibitem{Ghanem91}
{\sc R.~Ghanem and P.~Spanos}, {\em Stochastic Finite Elements: A Spectral
  Approach}, Springer, New York, NY, 1~ed., 1991,
  \url{https://doi.org/10.1007/978-1-4612-3094-6}.

\bibitem{Hastings70}
{\sc W.~K. Hastings}, {\em {Monte Carlo} sampling methods using {Markov} chains
  and their applications}, Biometrika, 57 (1970), pp.~97--109,
  \url{https://doi.org/10.2307/2334940}.

\bibitem{Herty19}
{\sc M.~Herty and G.~Visconti}, {\em Kinetic methods for inverse problems},
  Kinetic \& Related Models, 12 (2019), pp.~1109--1130,
  \url{https://doi.org/10.3934/krm.2019042}.

\bibitem{Hohenbichler81}
{\sc M.~Hohenbichler and R.~Rackwitz}, {\em Non-normal dependent vectors in
  structural safety}, Journal of the Engineering Mechanics Division, 107
  (1981), pp.~1227--1238, \url{https://doi.org/10.1061/JMCEA3.0002777}.

\bibitem{Iglesias18}
{\sc M.~Iglesias, M.~Park, and M.~V. Tretyakov}, {\em Bayesian inversion in
  resin transfer molding}, Inverse Problems, 34 (2018), p.~105002,
  \url{https://doi.org/10.1088/1361-6420/aad1cc}.

\bibitem{Iglesias2013}
{\sc M.~A. Iglesias, K.~Law, and A.~M. Stuart}, {\em Ensemble {Kalman} methods
  for inverse problems}, Inverse Problems, 29 (2013), p.~045001,
  \url{https://doi.org/10.1088/0266-5611/29/4/045001}.

\bibitem{Johnson94}
{\sc N.~L. Johnson, S.~Kotz, and N.~Balakrishnan}, {\em Continuous Univariate
  Distributions}, vol.~1 of Wiley Series in Probability and Statistics, Wiley,
  2~ed., 1994.

\bibitem{Kalman60}
{\sc R.~E. Kalman}, {\em A new approach to linear filtering and prediction
  problems}, Journal of Basic Engineering, 82 (1960), pp.~35--45,
  \url{https://doi.org/10.1115/1.3662552}.

\bibitem{Katafygiotis08}
{\sc L.~S. Katafygiotis and K.~M. Zuev}, {\em Geometric insight into the
  challenges of solving high-dimensional reliability problems}, Probabilistic
  Engineering Mechanics, 23 (2008), pp.~208--218,
  \url{https://doi.org/10.1016/j.probengmech.2007.12.026}.

\bibitem{katsuki1994hyperspace}
{\sc S.~Katsuki and D.~M. Frangopol}, {\em Hyperspace division method for
  structural reliability}, Journal of Engineering Mechanics, 120 (1994),
  pp.~2405--2427,
  \url{https://doi.org/10.1061/(ASCE)0733-9399(1994)120:11(2405)}.

\bibitem{Koutsourelakis04}
{\sc P.~S. Koutsourelakis, H.~J. Pradlwarter, and G.~I. Schuëller}, {\em
  Reliability of structures in high dimensions, part {I}: algorithms and
  applications}, Probabilistic Engineering Mechanics, 19 (2004), pp.~409 --
  417, \url{https://doi.org/10.1016/j.probengmech.2004.05.001}.

\bibitem{Kroese13}
{\sc D.~P. Kroese, R.~Y. Rubinstein, and P.~W. Glynn}, {\em Chapter 2 - {The}
  cross-entropy method for estimation}, in Handbook of Statistics, C.~R. Rao
  and V.~Govindaraj, eds., vol.~31 of Handbook of Statistics, Elsevier, 2013,
  pp.~19--34, \url{https://doi.org/10.1016/B978-0-444-53859-8.00002-3}.

\bibitem{Latz20}
{\sc J.~Latz}, {\em On the well-posedness of {Bayesian} inverse problems},
  SIAM/ASA Journal on Uncertainty Quantification, 8 (2020), pp.~451--482,
  \url{https://doi.org/10.1137/19M1247176}.

\bibitem{Latz18}
{\sc J.~Latz, I.~Papaioannou, and E.~Ullmann}, {\em Multilevel sequential
  {Monte Carlo} for {Bayesian} inverse problems}, Journal of Computational
  Physics, 368 (2018), pp.~154--178,
  \url{https://doi.org/10.1016/j.jcp.2018.04.014}.

\bibitem{Law15}
{\sc K.~Law, A.~M. Stuart, and K.~Zygalakis}, {\em Data Assimilation: A
  Mathematical Introduction}, Texts in Applied Mathematics, Springer, Cham,
  1~ed., 2015, \url{https://doi.org/10.1007/978-3-319-20325-6}.

\bibitem{Li16}
{\sc R.~Li, V.~Prasad, and B.~Huang}, {\em {Gaussian} mixture model-based
  ensemble {Kalman} filtering for state and parameter estimation for a {PMMA}
  process}, Processes, 4 (2016), \url{https://doi.org/10.3390/pr4020009}.

\bibitem{McLachlan05}
{\sc G.~McLachlan and D.~Peel}, {\em Ml fitting of mixture models}, in Finite
  Mixture Models, John Wiley \& Sons, Ltd, 2000, ch.~2, pp.~40--80,
  \url{https://doi.org/10.1002/0471721182.ch2}.

\bibitem{Melchers18}
{\sc R.~E. Melchers and A.~T. Beck}, {\em Structural Reliability Analysis and
  Prediction}, John Wiley \& Sons, Ltd, 3~ed., 2017,
  \url{https://doi.org/10.1002/9781119266105}.

\bibitem{Morio15}
{\sc J.~Morio and M.~Balesdent}, {\em Estimation of Rare Event Probabilities in
  Complex Aerospace and other Systems}, Woodhead Publishing, 1~ed., 2015,
  \url{https://doi.org/10.1016/C2014-0-02344-1}.

\bibitem{Peres10}
{\sc P.~M\"orters and Y.~Peres}, {\em Brownian Motion}, Cambridge Series in
  Statistical and Probabilistic Mathematics, Cambridge University Press, 2010,
  \url{https://doi.org/10.1017/CBO9780511750489}.

\bibitem{Nakagami60}
{\sc M.~Nakagami}, {\em The m-distribution, a general formula of intensity
  distribution of rapid fading}, in Statistical Methods in Radio Wave
  Propagation, W.~Hoffman, ed., Pergamon, 1960, pp.~3--36,
  \url{https://doi.org/10.1016/B978-0-08-009306-2.50005-4}.

\bibitem{Owen13}
{\sc A.~B. Owen}, {\em {Monte Carlo} theory, methods and examples}, 2013,
  \url{https://statweb.stanford.edu/~owen/mc/}.

\bibitem{Papaioannou15}
{\sc I.~Papaioannou, W.~Betz, K.~Zwirglmaier, and D.~Straub}, {\em {MCMC}
  algorithms for subset simulation}, Probabilistic Engineering Mechanics, 41
  (2015), pp.~89--103, \url{https://doi.org/10.1016/j.probengmech.2015.06.006}.

\bibitem{Papaioannou19}
{\sc I.~Papaioannou, S.~Geyer, and D.~Straub}, {\em Improved cross
  entropy-based importance sampling with a flexible mixture model}, Reliability
  Engineering \& System Safety, 191 (2019), p.~106564,
  \url{https://doi.org/10.1016/j.ress.2019.106564}.

\bibitem{Papaioannou16}
{\sc I.~Papaioannou, C.~Papadimitriou, and D.~Straub}, {\em Sequential
  importance sampling for structural reliability analysis}, Structural Safety,
  62 (2016), pp.~66--75, \url{https://doi.org/10.1016/j.strusafe.2016.06.002}.

\bibitem{Rackwitz01}
{\sc R.~Rackwitz}, {\em Reliability analysis—a review and some perspectives},
  Structural Safety, 23 (2001), pp.~365--395,
  \url{https://doi.org/10.1016/S0167-4730(02)00009-7}.

\bibitem{Reich19}
{\sc S.~Reich and S.~Weissmann}, {\em {Fokker--Planck} particle systems for
  {Bayesian} inference: computational approaches}, SIAM/ASA Journal on
  Uncertainty Quantification, 9 (2021), pp.~446--482,
  \url{https://doi.org/10.1137/19M1303162}.

\bibitem{Rubinstein16}
{\sc R.~Y. Rubinstein and D.~P. Kroese}, {\em Simulation and the {Monte Carlo}
  Method}, Wiley Series in Probability and Statistics, John Wiley \& Sons, Ltd,
  3~ed., 2016, \url{https://doi.org/10.1002/9781118631980}.

\bibitem{Schillings16}
{\sc C.~Schillings and A.~M. Stuart}, {\em Analysis of the ensemble {Kalman}
  filter for inverse problems}, SIAM Journal on Numerical Analysis, 55 (2017),
  pp.~1264--1290, \url{https://doi.org/10.1137/16M105959X}.

\bibitem{Schillings18}
{\sc C.~Schillings and A.~M. Stuart}, {\em Convergence analysis of ensemble
  {Kalman} inversion: the linear, noisy case}, Applicable Analysis, 97 (2018),
  pp.~107--123, \url{https://doi.org/10.1080/00036811.2017.1386784}.

\bibitem{Smith07}
{\sc K.~W. Smith}, {\em Cluster ensemble {Kalman} filter}, Tellus A: Dynamic
  Meteorology and Oceanography, 59 (2007), pp.~749--757,
  \url{https://doi.org/10.1111/j.1600-0870.2007.00246.x}.

\bibitem{Stuart10}
{\sc A.~M. Stuart}, {\em Inverse problems: a {Bayesian} perspective}, Acta
  Numerica, 19 (2010), p.~451–559,
  \url{https://doi.org/10.1017/S0962492910000061}.

\bibitem{Tukey77}
{\sc J.~W. Tukey}, {\em Exploratory Data Analysis}, Addison-Wesley Publishing
  Company, 1977.

\bibitem{Ullmann15}
{\sc E.~Ullmann and I.~Papaioannou}, {\em Multilevel estimation of rare
  events}, SIAM/ASA Journal on Uncertainty Quantification, 3 (2015),
  pp.~922--953, \url{https://doi.org/10.1137/140992953}.

\bibitem{Uribe20}
{\sc F.~Uribe, I.~Papaioannou, Y.~M. Marzouk, and D.~Straub}, {\em
  Cross-entropy-based importance sampling with failure informed dimension
  reduction for rare event simulation}, SIAM/ASA Journal on Uncertainty
  Quantification, 9 (2021), pp.~818--847,
  \url{https://doi.org/10.1137/20M1344585}.

\bibitem{waarts2000structural}
{\sc P.~H. Waarts}, {\em Structural reliability using finite element methods.
  An appraisal for {DARS}: Directional Adaptive Response surface Sampling},
  Delft University Press, 2000.

\bibitem{Wagner20}
{\sc F.~Wagner, J.~Latz, I.~Papaioannou, and E.~Ullmann}, {\em Multilevel
  sequential importance sampling for rare event estimation}, SIAM Journal on
  Scientific Computing, 42 (2020), pp.~A2062--A2087,
  \url{https://doi.org/10.1137/19M1289601}.

\bibitem{Wagner20_2}
{\sc F.~Wagner, J.~Latz, I.~Papaioannou, and E.~Ullmann}, {\em Error analysis
  for probabilities of rare events with approximate models}, SIAM Journal on
  Numerical Analysis, 59 (2021), pp.~1948--1975,
  \url{https://doi.org/10.1137/20M1359808}.

\bibitem{Wang16}
{\sc Z.~Wang and J.~Song}, {\em Cross-entropy-based adaptive importance
  sampling using von {Mises--Fisher} mixture for high dimensional reliability
  analysis}, Structural Safety, 59 (2016), pp.~42--52,
  \url{https://doi.org/10.1016/j.strusafe.2015.11.002}.

\end{thebibliography}

\end{document}